\documentclass{amsart}

\usepackage[dvips]{epsfig}
\usepackage{graphicx}
\usepackage{latexsym}
\usepackage{amsmath}
\usepackage{amsthm}
\usepackage{amssymb}
\usepackage{stmaryrd}
\RequirePackage[colorlinks=true,citecolor=blue,urlcolor=blue]{hyperref}

\usepackage{mathtools}

\usepackage{multirow}

 \usepackage[applemac]{inputenc}

\newtheorem{thm}{Theorem}
\newtheorem{assumption}[thm]{Assumption}
\newtheorem{lem}[thm]{Lemma}

\newtheorem{defi}[thm]{Definition}

\newtheorem{prop}[thm]{Proposition}
\newtheorem{rk}[thm]{Remark}

\newcommand{\p}{\partial}

\newcommand{\vip}{\vskip.2cm}

\newcommand{\field}[1]{\mathbb{#1}}
\newcommand{\EE}{\field{E}}

\newcommand{\PP}{\field{P}}

\newcommand{\RR}{\field{R}}
\newcommand{\TT}{\field{T}}

\newcommand{\Aa}{{\mathcal A}}

\newcommand{\Cc}{{\mathcal C}}

\newcommand{\Gg}{{\mathcal G}}

\newcommand{\Pp}{{\mathcal P}}

\newcommand{\Ss}{{\mathcal S}}
\newcommand{\Tt}{{\mathcal T}}

\newcommand{\Vv}{{\mathcal V}}

\newcommand{\Xx}{{\mathcal X}}

\def \ep {\varepsilon}

\def \a {\alpha}
\def \b {\beta}

\begin{document}

\title[Influence of variability on the Malthus parameter]{How does variability in cell aging and growth rates influence the Malthus parameter?}

\author{Ad\'ela\"ide Olivier}

\address{Ad\'ela\"ide Olivier, Universit\'e Paris-Dauphine, PSL Research University, CNRS, UMR [7534], CEREMADE, 75016 Paris, France.}

\email{adelaide.olivier@ceremade.dauphine.fr}

\begin{abstract}
The aim of this study is to compare the growth speed of different cell populations measured by their Malthus parameter.
We focus on both the age-structured and size-structured equations.
A first population (of reference) is composed of cells all aging or growing at the same rate $\bar v$.
A second population (with variability) is composed of cells each aging or growing at a rate $v$ drawn according to a non-degenerated distribution $\rho$ with mean $\bar v$.
In a first part, analytical answers -- based on the study of an eigenproblem -- are provided for the age-structured model.
In a second part, numerical answers -- based on stochastic simulations -- are derived for the size-structured model.
It appears numerically that the population with variability proliferates more slowly than the population of reference (for experimentally plausible division rates).
The decrease in the Malthus parameter we measure, around $2\%$ for distributions $\rho$ with realistic coefficients of variations around 15-20\%, is determinant since it controls the {\it exponential} growth of the whole population.
\end{abstract}

\maketitle









\section{Introduction}

Recent biological studies draw attention to the question of variability between cells. We refer to the study of Kiviet {\it et al.} published in 2014 \cite{4_Kiviet}.
A cell in a controlled culture grows at a constant rate $v>0$, but this rate can differ from one individual to another.
The biological question we address here states as follows. How does individual variability in the growth rate influence the growth speed of the population? 
The growth speed of the population is measured by the Malthus parameter we define thereafter, also called in the literature {\it fitness}.
Even if the variability in the growth rate among cells is small, with a distribution of coefficient of variation around 10\%,
and even if its influence on the Malthus parameter would be still smaller, such an influence may become determinant
since it characterises the exponential growth speed of the population.

\subsection{Deterministic modeling}

\subsubsection{Paradigmatic age or size equations}
Structured models have been successfully used to describe the evolution of a population of cells over the past decades, we refer to Metz and Diekmann \cite{4_MD}, the textbook of Perthame \cite{4_Perthame} and references therein. We focus in this study on classical structuring variables which are age (understood in a broad sense as a physiological age -- it may not be the time elapsed since birth) or size. 
The concentration $n(t,a,x)$ of cells of physiological age $a\geq 0$ and size $x>0$ at time $t\geq 0$ satisfies
\begin{equation} \label{4_eq:PDEdeltanovar}
\left\{
\begin{array}{l}
 \tfrac{\partial}{\partial t}n(t,a,x) + \frac{\partial}{\partial x} \big(g_x(a,x) n(t,a,x)\big) + \tfrac{\partial}{\partial a} \big(g_a(a,x) n(t,a,x)\big) + \gamma(x,a)n(t,a,x) = 0, \\ \\
g_a(a=0,x)n(t, a = 0, x) = 4 \int_0^\infty \gamma(a,2x)n(t,a,2x) da, \\ \\
g_x(a,x=0) n(t,a, x= 0) = 0, \quad n(t = 0, a, x) = n^{{\rm in}}(a,x),
\end{array}
\right.
\end{equation}
in a weak sense.
The mechanism at work here can be described as a mass balance. The concentration of cells $n(t,a,x)$ evolves through two terms of transport which stand for the growth and the aging, and with one term of fragmentation:
\begin{itemize}
\item[-] {\it Transport terms.} Both size and age evolve in a deterministic way, the growth speed and aging speed being respectively $g_x$ and $g_a$. Let $x_t$ and $a_t$ be the size and physiological age at time $t$ of a cell born with characteristics $(0,x_b) = (a_{t=0},x_{t=0})$. Then their evolution is given by $dx_t/dt = g_x(a_t, x_t)$ and $da_t/dt = g_a(a_t, x_t)$. 
\item[-] {\it Fragmentation terms.} We assume that the division is perfectly symmetric: an individual of size $x$ divides into two individuals of size $x/2$ (and age $0$). A cell of current size $x$ and current physiological age $a$ divides with probability $\gamma(a,x) dt$ between $t$ and $t+dt$. The boundary condition for $a=0$ ensures that the quantity of new cells of characteristics $(0,x)$ is exactly twice the number of cells of characteristics $(s,2x)$ for $s>0$ that have just divided. The factor $4$ is thus the product of a factor $2$ arising from the birth of two new cells at each division and another factor $2$ coming from the $2x$ size of dividing cells.\\
\end{itemize}

\noindent  Equation \ref{4_eq:PDEdeltanovar} also encloses two paradigmatic equations.
{\bf 1)} If $g_x$ and $\gamma$ do not depend on size, integrating \eqref{4_eq:PDEdeltanovar} in $x$ from zero to infinity,  we obtain the {\it age-structured equation}. The age-structured equation is a classical equation and we refer to Rubinov \cite{rubinov} and to Perthame \cite{4_Perthame} (Section 3.9.1) for a complete study. 
{\bf 2)} If $g_x$ and $\gamma$ do not depend on age, integrating \eqref{4_eq:PDEdeltanovar} in $a$ from zero to infinity,  we obtain the {\it size-structured equation}, as introduced by Metz and Diekmann \cite{4_MD}. We refer to Mischler and Scher \cite{4_MischlerScher} and to references therein for the study of this equation.

\subsubsection{Introducing individual variability}
To take into account variability between cells, we extend the previous framework adding an individual feature -- an aging rate or a growth rate -- as structuring variable.
Let $\Vv$ be a compact set of $(0,\infty)$ where the individual feature takes its values and
let  $\rho(v,dv')$ be a Markov kernel with support in $\Vv\times \Vv$ ({\it i.e.} satisfying $\int_\Vv \rho(v,dv')=1$ for all $v\in \Vv$).
In this study we focus on two cases of special interest: \\

\noindent {\bf Model (A+V).}  The age-structured model with variability, extending the classical setting (see Section~\ref{4_sec:age}).
The individual feature $v$ we introduce represents here an aging rate.
The concentration $n(t,a,v)$ of cells of physiological age $a\geq 0$ and aging rate $v\in \Vv$ at time $t\geq 0$ evolves as
\begin{equation*} \label{4_eq:edp_agevar}
\left\{
\begin{array}{l}
\frac{\partial}{\partial t} n(t,a,v) + \frac{\partial}{\partial a}\big(g_a(a,v) n(t,a,v)\big) + \gamma(a,v)n(t,a,v) = 0,  \\ \\  
g_a(a=0,v') n(t,a=0, v') = 2 \iint_{\Ss} \gamma(a,v) n(t,a,v)  \rho(v,dv') dv da, \quad \quad 
n(t=0,a,v)=  n^{{\rm in}}(a,v),
\end{array}
\right.
\end{equation*}
 in a weak sense, abusing slightly notation here,  $\Ss$ being the state space $[0,\infty)\times \Vv$. \\

\noindent {\bf Model (S+V).} The size-structured model with variability, as already introduced in Doumic, Hoffmann, Krell and Robert \cite{4_DHKR1} (see Section~\ref{4_sec:size}).
The concentration $n(t,x,v)$ of cells of size $x \geq 0$ and growth rate $v\in \Vv$ at time $t\geq 0$ evolves as
\begin{equation*} \label{4_eq:edp_sizevar}
\left\{
\begin{array}{l}
\frac{\partial}{\partial t} n(t,x,v) + \frac{\partial}{\partial x}\big(g_x(x,v) n(t,x,v)\big) + \gamma(x,v)n(t,x,v) = 4 \int_\Vv \gamma(2x,v') n(t,2x,v') \rho(v',v) dv', \\ \\
g_x(x=0,v') n(t,x=0, v') = 0, \quad \quad
n(t=0,x,v)=  n^{{\rm in}}(x,v),
\end{array}
\right.
\end{equation*}
still in a weak sense.\\

The underlying mechanism of the two previous equations is similar to the one described previously for~\eqref{4_eq:PDEdeltanovar}:
\begin{itemize}
\item[-] {\it Transport terms.}  The individual feature $v$ does not evolve through time, it is given once and for all at birth. Only the physiological age or the size evolve deterministically through the transport terms in $g_a$ or $g_x$. 
\item[-] {\it Fragmentation terms.} A cell of current physiological age $a$ or current size $x$, and feature~$v$, divides with probability $\gamma(a,v)dt$ or $\gamma(x,v)dt$ between $t$ and $t+dt$. It gives birth to two new cells, each of one having age $0$ or size $x/2$, and feature $v'$ with probability $\rho(v,dv')$.\\
\end{itemize}

  We stress again that the age-structured and size-structured equations were extensively studied but mainly without variability.
Here the novelty is the structuring variable $v$ we add to take into account an individual feature. \\

\subsection{Objective and related studies}

\subsubsection{Mathematical formulation of our problem}
The initial biological question can now be reformulated mathematically. As we will see, the Malthus parameter is defined as the dominant eigenvalue of  {\bf Model (A+V)} or {\bf (S+V)}. We denote it by $\lambda_{\gamma,\rho}$ to stress the dependence not only on the division rate $\gamma$ but also on the variability Markov kernel $\rho(v,v')dv'$. The long-time behaviour of the solution to {\bf Model (A+V)} or {\bf (S+V)} is expected to be $n(t,a,v) \approx e^{\lambda_{\gamma,\rho} t} N_{\gamma,\rho}(a,v)$ or $n(t,x,v) \approx e^{\lambda_{\gamma,\rho} t} N_{\gamma,\rho}(x,v)$, with $N_{\gamma,\rho}$ a stationary profile, where we see that the growth speed of the system is governed by $\lambda_{\gamma,\rho}$. 
Our aim is to compare the growth speed of the two following populations:\\ \label{page:pop}

\noindent {\bf 1)} {\bf Population of reference (without variability).}  All cells age or grow at the same rate~$\bar v$. This population grows at speed\footnote{We denote $\lambda_{\gamma,\rho}$ by $\lambda_{\gamma,\bar v}$ when $\rho$ is the Dirac mass at point $\bar v$.} $\lambda_{\gamma, \bar v}$. \\

\noindent {\bf 2)} {\bf Population with variability.} Each cell has its own aging or growth rate drawn according to a non-degenerated density $\rho(v')dv'$ (if one neglects heredity in the transmission of the individual feature). This population grows at speed $\lambda_{\gamma, \rho}$.\\

\noindent In other words, our aim is to compare $\lambda_{\gamma,\rho}$ to $\lambda_{\gamma,\bar v}$. Before all, one has to link $\rho$ and $\bar v$ in a biologically pertinent way, and this point is crucial.
To define properly an average aging or growth rate $\bar v$ can be done by several ways. Our choice in this article is to set $\bar v$ as the average growth rate of newborn cells, that is to say
\begin{equation} \label{eq:choixvbar}
\bar v = \frac{\int_{\mathcal V} v  \big( \iint_{[0,\infty)\times\mathcal V} \rho(v',v) \gamma(y,v') N_{\gamma,\rho}(y,v') dydv'\big) dv}{\iint_{[0,\infty)\times\mathcal V} \gamma(y,v) N_{\gamma,\rho}(y,v) dydv} = \int_{\mathcal V} v\rho(v) dv
\end{equation}
as soon as we neglect heredity in the transmission of variability.
I stress that other choices could be done to define $\bar v$: the average rate over any cell of the population or the average growth rate of dividing cells for instance\footnote{Respectively given by $\iint v N_{\gamma,\rho}(y,v)dydv$ and $\iint v\gamma(y,v) N_{\gamma,\rho}(y,v)dydv \big/ \iint \gamma(y,v) N_{\gamma,\rho}(y,v)dydv$.}. 
Another possibility would be to preserve the mean lifetime\footnote{One can check this is equivalent -- in both models -- to preserve the harmonic average of the density  $\rho$. So one should set $\frac{1}{\bar v} = \int_{\mathcal V} \frac{1}{v}\rho(v) dv$.}.
The study of all these possibilities, as interesting it is, lies beyond the aim of the paper and I decide to focus mainly on the choice (\ref{eq:choixvbar}) for a first study. 
Indeed the choice (\ref{eq:choixvbar}) is motivated by the fact that  it somehow enables us to preserve some aging or growing potential at birth between the two populations we compare.

\subsubsection{Existing studies on the fitness}

Before going ahead let us describe existing studies on the variations of the Malthus parameter.
We sum up below two studies dealing with the size-structured model (without variability).
We also mention a corpus of articles by Clairambault, Gaubert, Lepoutre, Michel and Perthame (one can see \cite{4_PerronFloquet2,4_GaubertLepoutre} and references therein, and also \cite{4_TLepoutre}) comparing the eigenvalue of a partial differential equation with time dependent periodic coefficients (birth and death rates)  to the eigenvalue of the equation with time-averaged coefficients. They mainly focus on the age-structured system for the cell division cycle. Evolution equations with time periodic coefficients require specific techniques (Floquet's theory).\\

\paragraph{{\it Asymmetry of the division.}}

The influence of asymmetry in the division is investigated by Michel in \cite{4_Michel06} for the model described by
\begin{equation*} \label{4_eq:size_pde_asym}
\tfrac{\partial}{\partial t} n(t,x) + \tfrac{\partial}{\partial x}n(t,x) + B(x)n(t,x)  = \tfrac{1}{\sigma}B(\tfrac{x}{\sigma})n(t,\tfrac{x}{\sigma}) + \tfrac{1}{1-\sigma}B\big(\tfrac{x}{1-\sigma}\big)n\big(t,\tfrac{x}{1-\sigma}\big),
\end{equation*}
with $n(t,x=0)=0$, where the size of each cell evolves linearly (at speed $g_x \equiv 1$) and a cell of size $x$ divides at a rate $B(x)$ into two cells of sizes $\sigma x$ and $(1-\sigma)x$, for some parameter $\sigma\in(0,1)$. 
The case of reference is the symmetric division case corresponding to $\sigma = 1/2$.
If we denote by $\lambda_{B,\sigma}$ the Malthus parameter in this model, the aim is to compare $\lambda_{B,\sigma}$ to $\lambda_{B,1/2}$. Depending on the form of the division rate $B$, the asymmetry division is beneficial or not for the growth of the overall cell population. 
Two special cases are highlighted in \cite{4_Michel06}, 
{\bf 1)} qualitatively, if cells divide at high sizes (the support of $B$ is far away from zero), then $\lambda_{B,\sigma} < \lambda_{B,1/2}$ which means that asymmetry slows the growth speed of the population,  {\bf 2)} qualitatively, if cells divides early (the support of $B$ contains zero and $B$ decreases), then $\lambda_{B,\sigma} >  \lambda_{B,1/2}$ which means that asymmetry creates a gain speeding up the growth of the population. 
We refer to Theorems~2.1 and~2.2 of \cite{4_Michel06} for precise statements of the assumptions on $B$. 
The method developed in \cite{4_Michel06} is an interesting approach to investigate our question as regards variability.\\

\paragraph{{\it Influence of the growth rate.}}
The influence of the individual growth rate is investigated by Calvez {\it et al.} in \cite{4_CDG12} for the model described by
\begin{equation*} \label{4_eq:size_pde}
\tfrac{\partial}{\partial t} n(t,x) + \tfrac{\partial}{\partial x}\big(\bar v x n(t,x)\big) + B(x)n(t,x) = 4 B(2x) n(t,2x),
\end{equation*}
where the size of each cell evolves exponentially at rate $\bar v$ and assuming that division is symmetric for simplicity.
On this very simple model, the Malthus parameter is exactly equal to the common individual growth rate $\bar v$, hence it \emph{seems} that the faster the cells grow, the faster the overall cell population grows. This is \emph{not} the case in general: if the growth speed would be $g(x)$ instead of $x,$ then increasing $g$ by a factor may have the effect of diminishing the Malthus parameter. 
A plausible example: if around infinity $g(x)$ is equivalent to $x^{\nu}$ (up to a constant) with $\nu<1$ and if $B(x)$ vanishes at infinity, being equivalent to $x^{-\gamma}$ (up to a constant) with $0< \gamma < 1-\nu$, then the Malthus parameter vanishes when inflating $g$ by a multiplicative factor.
This was proved in \cite{4_CDG12} (see Theorem~1 and Proposition~1 in Appendix~2, see also Figure~2(a)).
Note that in the present study, we inflate the growth speed by a factor that depends on each individual.

To finish I mention a recent study by Campillo, Champagnat and Fritsch \cite{Coralie}: for growth-fragmentation-death models, they focus on the variations of the first eigenvalue with respect to a parameter involved in both the growth speed and the birth and death rates, with a nice mixing of deterministic and stochastic techniques.

\subsection{Main results and outline}
 Our main results are summed up in Table~\ref{4_tab:resultats}. Our analysis begins with {\bf Model} {\bf (A+V)} in  Section~\ref{4_sec:age}, neglecting heredity in the transmission of the aging rate by picking $\rho(v,dv') = \rho(v')dv'$. We analytically prove that the Malthus parameter, well-defined by Theorem~\ref{4_thm:ageeig_general} (we give a proof for the sake of completeness in Section \ref{sec:proofvp}), can increase or decrease when introducing variability in the aging rate, depending on the form of the division rate $\gamma$ 
 (Theorem~\ref{4_thm:ageinfluence}). We also compute the perturbation at order two of the Malthus parameter $\lambda_{\gamma,\rho_\a}$ when $\rho_\a$ converges in distribution to a Dirac mass as $\a \rightarrow 0$ (Theorem~\ref{4_thm:ageperturbation}).

The paradigmatic age-structured model still being a toy model for cellular division (at least for {\it E. coli}, see Robert {\it et al.} \cite{4_DHKR2}), we turn to {\bf Model} {\bf (S+V)} in Section~\ref{4_sec:size}. A numerical study, based on stochastic simulations, is carried out to give preliminary answers. For the division rate $\gamma$ chosen as a power law with some lag (see Table~\ref{4_tab:resultats}), for $\rho(\cdot)$ a truncated Gaussian distribution with mean $\bar v$, we infer that 
$\lambda_{\gamma,\rho} < \lambda_{\gamma,\bar v}$. Admittedly this conclusion agrees with biological wisdom but the strength of our methodology is to quantify such a decrease. 
We evaluate the magnitude of the decrease around 2\% when the variability distribution has a realistic coefficient of variation around 15\%-20\%. Such a decrease is far from being negligible since the Malthus parameter governs the {\it exponential} growth of the whole population. In addition, we observe a monotonous relationship: the Malthus parameter decreases when there is more and more variability in the growth rate.

{\small
\begin{table}[h!]
\centering
{\small
\begin{tabular}{cccc}
\hline \hline
{\bf }                                        					& {\bf Model} {\bf (A+V)}                                         	& {\bf Model} {\bf (S+V)} 							\\
Division rate $\gamma$                        			& $v(a-1)^2 {\bf 1}_{\{a\geq 1\}}$            		&       $vx (x-1)^2 {\bf 1}_{\{x\geq 1\}}$           	                	\\
Variability                              					& {\small In the aging rate}                                	& {\small In the growth rate} 						\\ \\
\multirow{2}{*}{Variations} 					& $\lambda_{\gamma,\rho} \searrow$            	& $\lambda_{\gamma,\rho} \searrow$                  		\\ 
                                               					& (Figure~\ref{4_fig:AgeVar_Bpuissance})    	& (Figure~\ref{4_fig:SizeVar_lambdaalpha}) 			\\  \\
\multirow{3}{*}{{\it Comments}} 					& {\it Analytic result:} 					& \multirow{2}{*}{--} 								\\
                                                   					& {\it Theorems~\ref{4_thm:ageinfluence} and~\ref{4_thm:ageperturbation}}  					& 	\\ \hline \hline
\end{tabular}
\caption{{\it Variations of the Malthus parameter compared to the reference value when introducing variability between cells, for an experimentally realistic division rate. Note: $\lambda_{\gamma,\rho} \searrow$ means that $\lambda_{\gamma,\rho} < \lambda_{\gamma,\bar v}$ for a non-degenerated probability distribution $\rho(\cdot)$ with mean $\bar v$ (truncated Gaussian), and so on.
} \label{4_tab:resultats}}
}
\end{table}
}
We provide perspectives in Section~\ref{4_sec:discussion}.
 Section~\ref{4_sec:proofs} is devoted to the proofs of our main results of Section~\ref{4_sec:age}. Section~\ref{4_sec:appendix} gives supplementary figures and tables to complete the numerical study of Section~\ref{4_sec:size}. Finally Sections \ref{sec:proofvp} and \ref{sec:supp} are two appendices.

\section{The age-structured model with variability} \label{4_sec:age}

 In this section we study the age-structured model with variability {\bf (A+V)}. 
Lebowitz and Rubinow \cite{4_age2} and Rotenberg \cite{4_age1} already enriched the age-structured  model by an additional feature. 
In \cite{4_age2}, the model is structured by the age and the lifetime, called generation time, which is inherited from the mother cell to the daughter cells. The death or disappearance of the cells occur at a constant rate. 
In \cite{4_age1}, the model is structured by a maturity (our physiological age) and by a maturity velocity (our aging rate). 
Rotenberg's model is further studied by Mischler, Perthame and Ryzhik \cite{4_age3}, establishing the existence of a steady state and the long-time behaviour of the solution. 

Rotenberg's model is close to the model we are interested in, but different. 
From a certain viewpoint, the model introduced by Rotenberg \cite{4_age1} and studied in \cite{4_age3} is more general than {\bf Model (A+V)}: there are distinct death and birth rates and maturity velocity can change at any moment during the life of an individual. However,  in our {\bf Model (A+V)}  we allow for a general evolution of the aging rate (through the aging speed $g_a$) and, most importantly, we allow for heredity in the transmission of the aging rate (through the Markov kernel $\rho$) in Section \ref{subsec:agedefMalthus} which follows.

\subsection{Definition of the Malthus parameter} \label{subsec:agedefMalthus}

\subsubsection{Main assumptions}

 We use the notation $\Ss = [0,\infty)\times\Vv$ for the state space of the physiological age and the aging rate.
We require the following two assumptions throughout this section.
\begin{assumption}[Aging speed $g_a$ and division rate $\gamma$] \label{4_ass:age_basic}
The following conditions are fulfilled:
\begin{itemize}
\item[{\bf (a)}] Both $(a,v)\leadsto \gamma(a,v)$ and $(a,v) \leadsto g_a(a,v)$ are uniformly continuous.
\item[{\bf (b)}]  For any $(a,v)\in \Ss$, $g_a(a,v) > 0$ and  $\sup_{(a,v)\in\Ss} g_a(a,v) < \infty$.
\item[{\bf (c)}] For any $v\in \Vv$ there exits $0\leq a_{\min}(v) < a_{\max}(v) \leq \infty$ such that $\gamma(a,v)>0$ for $a\in[a_{\min}(v),a_{\max}(v)]$. 
\item[{\bf (d)}] For any $v \in \Vv$, we have $\int^{\infty} \tfrac{\gamma(a,v)}{g_a(a,v)} da = \infty$ and
$
\sup_{(a,v)\in \Ss} \tfrac{\gamma(a,v)}{g_a(a,v)} e^{ - \int_0^a \frac{\gamma(s,v)}{g_a(s,v)} ds } < \infty.
$
\end{itemize}
\end{assumption}
\begin{assumption}[Markov kernel $\rho$] \label{4_ass:age_rho} There exists a continuous and bounded $\rho : \Vv^2 \rightarrow [0,\infty)$ satisfying $\int_\Vv \rho(v,v')dv'=1$ for any $v\in\Vv$, such that $\rho(v,dv') = \rho(v,v') dv'$, which satisfies in addition
$$
\inf_{v'\in \Vv}\int_\Vv \rho(v,v') dv = \varrho > \tfrac{1}{2}.
$$
\end{assumption}

These two assumptions enable us to define the Malthus parameter studying the direct and adjoint eigenproblems below.
In Assumption~\ref{4_ass:age_basic}, we require {\bf (a)} since we look for continuous eigenvectors.
The requirement {\bf (b)} on the aging speed $g_a$ authorises the case $g_a(a,v) = v$ we are interested in (Section \ref{4_sec:age_influence}). Note that the non-negativity of $g_a$ means that cells can only age, but not rejuvenate. The boundedness of the aging speed is technical.
The condition {\bf (c)} is needed in our proof for the uniqueness of the eigenelements.
Finally note that {\bf (d)} is not very restrictive and it can be read as follows: for any $v\in \Vv$, $a \leadsto \tfrac{\gamma(a,v)}{g_a(a,v)}$ is a hazard rate and the associated density is bounded in $a$ (and in $v$, which belongs to a compact set of $(0,\infty)$).
Assumption~\ref{4_ass:age_rho} on $\rho$ is rather strong but it simplifies the study of the eigenproblem. 

These assumptions have the advantage of leading to a simple proof (in Section \ref{4_sec:proofs}). The recent results of Mischler and Scher \cite{4_MischlerScher} should enable us to weaken it. We stress that our objective here is not the extensive study of the eigenproblem, but lies beyond with the study of the variations of the eigenvalue.

\subsubsection{The direct and adjoint eigenproblems}

We now precisely define the Malthus parameter.
To that end, let us introduce the direct eigenproblem,
\begin{equation} \label{4_eq:EDP_N}
\left\{
\begin{array}{l}
 \frac{\partial}{\partial a} \big(g_a(a,v) N_{\gamma,\rho}(a,v) \big) + \gamma(a,v) N_{\gamma,\rho}(a,v) = - \lambda_{\gamma,\rho} N_{\gamma,\rho}(a,v), \\ \\ 
g_a(a= 0, v') N_{\gamma,\rho}(a = 0, v') = 2 \iint_{\Ss} \gamma(a,v) N_{\gamma,\rho}(a,v) \rho(v,v')  dv da, \\ \\
N_{\gamma,\rho} \geq 0, \quad \iint_\Ss N_{\gamma,\rho}(a , v) dvda = 1,
\end{array}
\right.
\end{equation}
and the adjoint eigenproblem,
\begin{equation} \label{4_eq:EDP_Psi}
\left\{
\begin{array}{l}
g_a(a,v) \tfrac{\partial }{\partial a}\phi_{\gamma,\rho}(a,v) + \gamma(a,v) \Big( 2 \int_{\Vv} \phi_{\gamma,\rho}(0,v')  \rho(v,v') dv' - \phi_{\gamma,\rho}(a,v) \Big) = \lambda_{\gamma,\rho} \phi_{\gamma,\rho}(a,v),  \\ \\ 
\phi_{\gamma,\rho} \geq 0, \quad \iint_\Ss (N_{\gamma,\rho} \phi_{\gamma,\rho})(a,v) dvda = 1,
\end{array}
\right.
\end{equation}
linked to {\bf Model (A+V)}.\\

Let $\Cc_b\big(\Vv)$  be the set of functions $f : \Vv \rightarrow \RR$ which are bounded and continuous and let $\Cc^1_b\big(\RR^+)$ be the set of functions $f : \RR^+\rightarrow \RR$  which are bounded and continuously differentiable. For $f : \Ss \rightarrow \RR$ the notation $f\in \Cc^1_b\big(\RR^+; \Cc_b(\Vv)\big)$ means that $a \leadsto f(a,v)$ belongs to $\Cc^1_b\big(\RR^+)$ for any $v\in \Vv$ and that $v \leadsto f(a,v)$ belongs to $\Cc_b\big(\Vv)$ for any $a\in\RR^+$.

\begin{thm} \label{4_thm:ageeig_general}
Work under Assumptions~\ref{4_ass:age_basic} and~\ref{4_ass:age_rho}.
There exists a unique solution $(\lambda_{\gamma,\rho},N_{\gamma,\rho},\phi_{\gamma,\rho})$ to the direct and adjoint eigenproblems \eqref{4_eq:EDP_N} and \eqref{4_eq:EDP_Psi} such that $\lambda_{\gamma,\rho} > 0$, $(g_a N_{\gamma,\rho}) \in \Cc^1_b\big(\RR^+; \Cc_b(\Vv)\big)$ and $\phi_{\gamma,\rho} \in \Cc^1_b\big(\RR^+; \Cc_b(\Vv)\big)$.
\end{thm}

The unique $\lambda_{\gamma,\rho}$ defined in such a way is what we call the Malthus parameter (or {\it fitness}) and let us now study its variations (with respect to $\rho$, for $\gamma$ fixed).

\subsection{Influence of variability on the Malthus parameter} \label{4_sec:age_influence}
A preliminary remark first: let us consider the case of a constant division rate, $\gamma(a,v) = c > 0$ for any $(a,v) \in \Ss$. Since $\gamma$ is constant, we expect that variability in the aging rate has no influence on the Malthus parameter. When $\rho(v,dv') = \rho(v')dv'$, one can easily check that $\lambda_{\gamma,\rho} = c$ which is independent of $\rho$.

\subsubsection{Model specifications}
From now on we consider {\bf Model (A+V)} with the aging speed set to
\begin{equation} \label{4_eq:gv}
g_a(a,v) = v
\end{equation}
which means that the physiological age is proportional to the time elapsed since birth, up to a factor $v$ which may change from an individual to another. \\

Let $0< a_{\max} \leq \infty$.
For a continuous $B: [0,a_{\max}) \rightarrow [0,\infty)$ such that $\int^{a_{\max}} B(s)ds = \infty$,
we assume that
\begin{equation} \label{4_eq:vBa}
\gamma(a,v) = g_a(a,v)B(a) = v B(a).
\end{equation}
Cells divides at a rate 
$$\gamma(a,v)  dt = g_a(a,v) B(a) dt = B(a) da,$$ 
since $da = g_a(a,v) dt$.
Thus one has to see the division rate $B(a)$ as a rate per unit of physiological age and $\gamma(a,v) = g_a(a,v)B(a)$  as a rate per unit  of time. We mimic here the choice made by S. Taheri-Araghi {\it et al.} \cite{4_VergassolaJun} in a more general model -- choice relying on biological evidence.\\

\subsubsection{Variations}
In this section, contrarily to the previous one, we neglect heredity in the transmission of the aging rate assuming that
\begin{equation} \label{4_eq:noheredity}
\rho(v,dv') = \rho(v') dv'
\end{equation}
for some continuous and bounded $\rho : \Vv \rightarrow [0,\infty)$ such that $\int_\Vv \rho(v')dv' = 1$.
In this framework\footnote{Note that under Specifications  \eqref{4_eq:gv} and \eqref{4_eq:vBa} the previous Assumption \ref{4_ass:age_basic} is valid.}, the eigenvectors solution to the eigenproblem \eqref{4_eq:EDP_N}--\eqref{4_eq:EDP_Psi} are explicit and an implicit relation uniquely defines the eigenvalue (see Lemma~\ref{4_lem:ageeigen} below, Section~\ref{4_sec:proofs}).\\

From now on we denote $\lambda_{\gamma,\rho}$ by $\lambda_{B,\rho}$ (to stress Specification \eqref{4_eq:vBa}) and $\lambda_{B,\bar v}$ stands for $\lambda_{B,\delta_{\bar v}}$ where $\delta_{\bar v}$ is the Dirac mass at point 
\begin{equation} \label{4_eq:meanbarv}
\bar v = \int_{\Vv} v \rho(v) dv.
\end{equation}
 If necessary, set $B(a) = 0$ for $a\geq a_{\max}$, and define
\begin{equation} \label{4_eq:fB}
\Psi_B(a) = B(a) \exp(-\int_0^a B(s) ds), \quad a\geq 0.
\end{equation} 
\vip

One wants to compare \\

\noindent {\bf 1)} the Malthus parameters $\lambda_{B,\bar v}$ solution to
\begin{equation} \label{4_eq:age_lambda_novar}
2 \int_0^\infty \exp \Big( - \frac{\lambda_{B,\bar v}a }{\bar v} \Big) \Psi_B(a) da = 1,
\end{equation}
controlling the growth speed of the {\bf population of reference} ; 
to\\

\noindent {\bf 2)} the Malthus parameters $\lambda_{B,\rho}$ solution to
\begin{equation} \label{4_eq:age_lambda}
2 \iint_{\Ss} \exp \Big( - \frac{\lambda_{B,\rho} a }{v} \Big) \Psi_B(a) \rho(v) dvda = 1,
\end{equation}
controlling the growth speed of the {\bf population with variability}. \\

\begin{thm} \label{4_thm:ageinfluence}
Consider {\bf Model (A+V)} with Specifications \eqref{4_eq:gv}, \eqref{4_eq:vBa}, \eqref{4_eq:noheredity} and \eqref{4_eq:meanbarv}.
Assuming in addition that $B : [0,a_{\max}) \rightarrow [0,\infty)$ is differentiable,
\begin{itemize}
\item[(i)] for $a_{\max}\in(0,\infty]$, if $B'(a) < B(a)^2$ for any $a \in [0,a_{\max})$ then $\lambda_{B,\rho} < \lambda_{B,\bar v}$. 
\item[(ii)] for $a_{\max}< \infty$, if $B$ is such that $\Psi_B$ is bounded and $B'(a) > B(a)^2$ for any $a\in [0,a_{\max})$ then $\lambda_{B,\rho} > \lambda_{B,\bar v}$.
\end{itemize}
\end{thm}

First note that Theorem~\ref{4_thm:ageinfluence} is valid independently of the density $\rho$ on $\Vv$ (with mean $\bar v$).
Also note that an acceleration of the proliferation with dispersion is not systematic. Indeed, when variability in the aging rates increases (at fixed mean), for some division rates $B$, the proliferation can slow down (statement (i)).

The first derivative of the density $\Psi_B$ being
$$
\Psi'_B(a) = \big(B'(a) - B(a)^2\big) \exp(-\int_0^a B(s) ds),
$$
(for $B$ differentiable), a condition of sign on $B'-B^2$ is nothing but imposing $\Psi_B$ to be decreasing (statement (i)) or increasing (statement (ii)). 
One canonical example is the following. Assume that $B$ is constant equal to $b>0$ so that $\gamma(a,v) = vb$  for any $(a,v) \in \Ss$. Then, $\Psi_B$ is the exponential density of mean $b^{-1}$: it is decreasing and statement (i) of Theorem~\ref{4_thm:ageinfluence} applies. \\

We give the proof of Theorem~\ref{4_thm:ageinfluence} now. Relying mainly on Jensen's inequality, we obtain a straightforward proof.

\begin{proof}[Proof of Theorem~\ref{4_thm:ageinfluence}]
Let $\lambda_{B,\bar v}$ be such that 
\begin{equation*} \label{4_lambda0}
H_{\bar v}(\lambda_{B,\bar v}) = 1 \quad \textrm{with} \quad H_{\bar v}(\lambda) = 2 \int_0^\infty \exp \Big( - \frac{\lambda a }{\bar v} \Big) \Psi_B(a) da ,
\end{equation*}
and $\lambda_{B,\rho}$ be such that
\begin{equation*} \label{4_lambdarho}
H_{\rho}(\lambda_{B,\rho}) = 1 \quad \textrm{with} \quad  H_{\rho}(\lambda) = 2 \iint_{\Ss} \exp \Big( - \frac{\lambda a }{v} \Big) \Psi_B(a) \rho(v) dvda.
\end{equation*}
We know that 
{\bf 1)} both $H_{\bar v}$ and $H_{\rho}$ are continuous on $[0,\infty)$, 
{\bf 2)} both $H_{\bar v}$ and $H_{\rho}$ are decreasing on $[0,\infty)$,
{\bf 3)}  when $\lambda$ goes to infinity, $H_{\bar v}(\lambda) \rightarrow 0$ and $H_{\rho}(\lambda) \rightarrow 0$, 
{\bf 4)} at point $\lambda = 0$, $H_{\bar v}(\lambda = 0) = H_{\rho}(\lambda=0) = 2 >1$. Then, by the intermediate values theorem, we know there exists a unique positive $\lambda_{B,\bar v}$ such that $H_{\bar v}(\lambda_{B,\bar v}) = 1$ and a unique positive $\lambda_{B,\rho}$ such that $H_{\rho}(\lambda_{B,\rho}) = 1$.
Recall that we want to compare $\lambda_{B,\rho}$ to $\lambda_{B,\bar v}$. We claim that {\bf 5)} if $\Psi'_B\leq 0$, then $$H_{\rho}(\lambda) < H_{\bar v}(\lambda)$$ for any $\lambda>0$. Together with points {\bf 1)} to {\bf 4)}, this enables us to conclude that  $\lambda_{B,\rho} < \lambda_{B,\bar v}$.

\vip
Let us prove {\bf 5)}.
For any $\lambda >0$, introduce
$$
h_\lambda : v \in \Vv \leadsto  2 \int_{0}^\infty \exp \Big( - \frac{\lambda a }{v} \Big) \Psi_B(a) da.
$$
Then $H_{\bar v}$ and $H_{\rho}$
can be written
$$
H_{\bar v}(\lambda) =  h_\lambda(\bar v) \quad \textrm{and} \quad H_{\rho}(\lambda) = \int_{\Vv}  h_\lambda(v)  \rho(v) dv
$$
for $\lambda>0$.
We claim that $h_\lambda$ is strictly concave when $\Psi'_B\leq 0$. Then $h_\lambda(\bar v) > \int_{\Vv} h_\lambda(v) \rho(v) dv$ by Jensen's inequality, since $\int_{\Vv} v \rho(v) dv = \bar v$, and {\bf 5)} immediately follows.  \\

We now compute the second derivative of $h_\lambda$,
$$
h_{\lambda}''(v) = \frac{2}{v^2} \int_0^\infty  \frac{\lambda a}{v} \Big(\frac{\lambda a}{v} - 2\Big) e^{-\frac{\lambda a}{v}}  \Psi_B(a) da.
$$
Integrating by parts, recalling that we choose $B$ differentiable, we get
\begin{equation*}
h_{\lambda}''(v) =  \frac{2 \lambda}{v^3} \int_0^{\infty} a^2 e^{-\frac{\lambda a}{v}} \Psi'_B(a) da < 0
\end{equation*}
for any $v \in \Vv$ when $\Psi'_B< 0$. This ends the proof for the case  $\Psi'_B\leq 0$ and the case $\Psi'_B > 0$ is treated in the same way (we need $\Psi_B$ to be bounded for the integration by parts).
\end{proof}

\subsubsection{Small perturbations}

The  density of the aging rates is thought of as a deformation around an average aging rate $\bar v$. 
Given a baseline density $\rho : \Vv \rightarrow [0,\infty)$ continuous and bounded with mean $\bar v$, we define, for $\alpha \in (0,1]$,
\begin{equation} \label{4_eq:rhoalpha}
\rho_\a(v) = \a^{-1}\rho\big(\a^{-1}(v- (1-\a) \bar v)\big), \quad v \in \Vv_\a = \a \Vv + (1 - \a) \bar v,
\end{equation}
so that $\int_{\Vv_\a} \rho_\a(v) dv=1$ and the mean value is constant, {\it i.e.} $\int_{\Vv_\a} v\rho_\a(v)dv=\bar v.$ The density $\rho_\a$ converges in distribution to the Dirac mass $\delta_{\bar v}$ as $\a\to 0.$
Thus we want to investigate the behaviour of the Malthus parameter $\lambda_{B,\rho_\a}$ -- which is a perturbation of $\lambda_{B,\bar v}$ -- for small value of $\a$.

\begin{thm} \label{4_thm:ageperturbation} 
Consider {\bf Model (A+V)} with Specifications \eqref{4_eq:gv}, \eqref{4_eq:vBa} and $\rho(v,dv') = \rho_\a(v')dv'$ with $\rho_\a$ defined in \eqref{4_eq:rhoalpha}.
Then $\a \leadsto \lambda_{B,\rho_\a}$ is twice differentiable and the Malthus parameter $\lambda_{B,\rho_\a}$ defined by \eqref{4_eq:age_lambda} satisfies 
$$
\lambda_{B,\rho_\a} = \lambda_{B,\bar v} + \frac{\a^2}{2} \, \frac{d^2 \lambda_{B,\rho_\a}}{d \a^2}\Big|_{\a=0} + o(\a^2)
$$
with $\lambda_{B,\bar v}$ defined by \eqref{4_eq:age_lambda_novar} and
$$
 \frac{d^2 \lambda_{B,\rho_\a}}{d\a^2} \Big|_{\a=0} = \sigma^2\Big( \int_0^{\infty} \frac{a}{\bar v} e^{ - \frac{\lambda_{B,\bar v} a }{\bar v} } \Psi_B(a) da \Big)^{-1} \int_0^\infty \frac{\lambda_{B,\bar v} a}{\bar v} \Big( \frac{\lambda_{B,\bar v} a}{\bar v} - 2  \Big) e^{- \frac{\lambda_{B,\bar v} a}{\bar v}} \Psi_B(a) da
$$
where $\sigma^2 = \int_\Vv (v-\bar v)^2 \rho(v) dv$.
\end{thm}

Note that the perturbation plays at order 2, but not at order 1 since the mean $\bar v$ is preserved (we prove $\frac{d \lambda_{B,\rho_\a}}{d \a}\big|_{\a=0} =0$ in Proposition~\ref{4_prop:age_d1} below, Section~\ref{4_sec:proofs}). The amplitude of the perturbation at order 2 depends on the baseline aging rate density $\rho$ only through its means $\bar v$ and its variance $\sigma^2$.
The proof heavily relies on the explicit expressions of the eigenelements solution to the eigenproblem \eqref{4_eq:EDP_N}--\eqref{4_eq:EDP_Psi}, available when neglecting heredity in the transmission of the aging rate (see Lemma~\ref{4_lem:ageeigen} below, Section~\ref{4_sec:proofs}).

\begin{rk} If $B$ is differentiable, integrating by parts we get
$$
 \frac{d^2 \lambda_{B,\rho_\a}}{d\a^2} \Big|_{\a=0} =
 \lambda_{B,\bar v} \sigma^2
 \Big( \int_0^{\infty} a e^{ - \frac{\lambda_{B,\bar v} a }{\bar v} } \Psi_B(a) da \Big)^{-1} 
 \int_0^\infty a^2 e^{- \frac{\lambda_{B,\bar v} a}{\bar v}} \Psi'_B(a) da
$$
and the variation study of $\a \leadsto \lambda_{B,\rho_\a}$ around $\a = 0$ follows as in the proof of Theorem~\ref{4_thm:ageinfluence} above.\\
\end{rk}

\vip

The class of $B$ for which Theorem~\ref{4_thm:ageinfluence} applies is unfortunately quite restrictive. In view of applications, a monotonous density $\Psi_B$ does not seem realistic since we expect to observe a mode in the density of the physiological ages at division (exactly as for lifetimes). One can see the distribution of lifetimes, also called {\it doubling times} or {\it generation times}, in \cite{4_osella} (Figure 1.D), and in \cite{4_VergassolaJun} for various medium (Figure S10). Given a more realistic division rate $B$, in this simple model, one can numerically solve Equations \eqref{4_eq:age_lambda} and \eqref{4_eq:age_lambda_novar} in order to compare $\lambda_{B,\rho}$ to $\lambda_{B,\bar v}$.

A plausible division rate would be of the form $B(a) = (a-1)^\b {\bf 1}_{\{a\geq 1\}}$. For such a division rate, Theorem~\ref{4_thm:ageinfluence} does not apply since $\Psi_B$ is non-monotonous.
Figure~\ref{4_fig:AgeVar_Bpuissance} shows curves $\a \leadsto \lambda_{B,\rho_\a}$ for different values of $\beta$ and we see that all these curves are non-increasing, which means that variability in the aging rate slows down the growth speed of the overall cell population, for such division rates~$B$ and the tested baseline variability kernel (a truncated Gaussian distribution).

\begin{figure}[h!]
\begin{center}
\includegraphics[width=13.5cm]{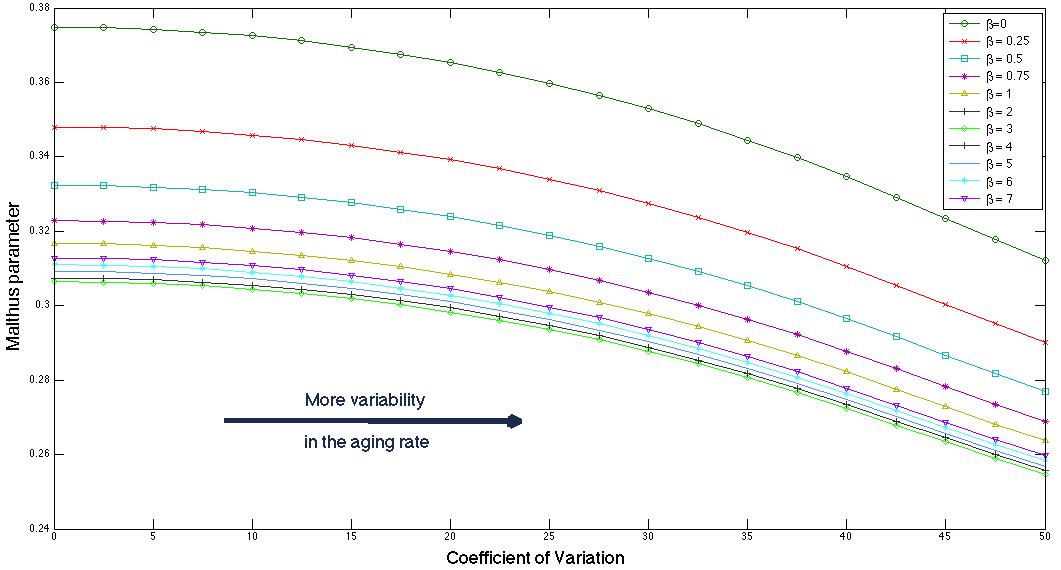}
\caption{{\it {\bf Model (A+V)}. $CV_{\rho_\a} \leadsto \lambda_{B,\rho_\a}$ defined by \eqref{4_eq:age_lambda} for $\rho_\a(v) = \a^{-1}\rho\big(\a^{-1}(v- (1-\a)\bar v)\big)$ (the baseline density $\rho$ is a Gaussian density with mean $\bar v = 1$ and standard deviation $0.7$ truncated on $[0,2]$) and different division rates $\gamma(a,v) = vB(a)$ with $B(a) = (a-1)^\beta {\bf 1}_{\{a\geq1\}}$, $\beta \in \{0, 0.25, 0.5, 0.75, 1, 2, \ldots, 7\}$. 
Reference (all cells age at rate $\bar v = 1$): point of null abscissa and y-coordinate $\lambda_{B,\bar v = 1}$ defined by \eqref{4_eq:age_lambda_novar}.} \label{4_fig:AgeVar_Bpuissance}} 
\end{center}
\end{figure}

\newpage

\section{The size-structured model with variability} \label{4_sec:size}

\noindent In this section we study the size-structured model with variability {\bf (S+V)}.
A precise definition of the Malthus parameter $\lambda_{\gamma,\rho}$ is based on the associated eigenproblem.
We refer to Doumic and Gabriel \cite{4_DG10} (size-structured model without variability) and references therein for the study of the eigenproblem.  We carry out a numerical study of our initial question. The purpose of the present section is twofold.\\

\noindent {\bf 1)} Exhibit first results regarding the influence of variability on the Malthus parameter, at least for biologically realistic parameters. \\

\noindent {\bf 2)} Show how one can use stochastic simulations to approximate numerically the Malthus parameter, and highlight the strength of this approach. Indeed one could use deterministic methods and approximate the Malthus parameter via a discretisation of the PDE. The stochastic approach we propose enables us to build confidence intervals for $\lambda_{\gamma,\rho}$ (in order to check if the variation compared to a reference value is significant or not).
In addition, the computational cost of our method should be invariant to the dimension of the model and it could be used as well for more complex models.

\newpage
\subsection{Stochastic approach} \label{subsec:size_sto}

\subsubsection{Stochastic modeling}

We closely follow here the description given in Doumic {\it et al.} \cite{4_DHKR1}.
Introduce the infinite genealogical tree
$$
\TT = \bigcup_{m=0}^{\infty} \{0,1 \}^m.
$$
Informally we may view $\TT$ as a population of cells (where the initial individual is denoted by $\emptyset$).
Each node $u \in \TT$, identified with a cell of the population, has a mark
$$
(b_u, \zeta_u, \xi_u , \tau_u),
$$ 
where $b_u$ is the birth time, $\zeta_u$ the lifetime, $\xi_u$ the size at birth and $\tau_u$ the individual feature of cell~$u$. We introduce $d_u$ the division time of $u$, $d_u = b_u + \zeta_u$.\\

The growth rate of one cell is drawn according to 
$$
\rho(v,dv') = \PP(\tau_u \in dv' | \tau_{u^-} = v)
$$
where $u^-$ is the parent of $u$.
We choose $g_x(x,v) = vx$ which means the size of each sizes evolves exponentially. Such an assumption is realistic, we refer to the analysis of single cells growth  by Schaechter {\it et al.} \cite{4_exponential} and by Robert {\it et al.} \cite{4_DHKR2} (Figure 2). Then the size at time $t$ of cell $u$ is
$$
\xi_u^t = \xi_u e^{\tau_u t} , \quad t \in [b_u,d_u).
$$
Stochastically, cells divide according to the rate $\gamma(x,v)dt$,
$$
\PP (\zeta_u \in [t,t+dt) | \zeta_u \geq t , \xi_u = x_b, \tau_u = v) = \gamma(x_b e^{vt},v)dt
$$
where $x_b e^{vt} = x$ is the size after a time $t$  for a cell born with size $x_b$ and growth rate $v$.
Cells divide into two equal parts,
\begin{equation} \label{4_eq:2equalparts}
\xi_{u} = \tfrac{1}{2} \xi_{u^-} \exp\big(\tau_{u^-}\zeta_{u^-}\big).
\end{equation}
From the two previous relations, we readily obtain that $(\xi_u,u\in \TT)$ is a bifurcating Markov chain (see \cite{4_Guyon} for a definition) with explicit transition
\begin{equation*} \label{4_eq:sizetransition}
 \Pp_\gamma(x,y)dy = \PP(\xi_u \in dy | \xi_{u^-} = x).
\end{equation*}

This stochastic modeling and the previous deterministic modeling match. 
Define the set of living cells at time $t$,
\begin{equation} \label{4_eq:livingset}
\partial \Tt_t =  \{ u \in \TT ; b_u \leq t < d_u \}.
\end{equation}
Define the measure $n(t,dxdv)$ as the expectation of the empirical measure at time $t$ over smooth test functions $f : \Ss \rightarrow \RR$,
$$
\iint_\Ss f(x,v) n(t,dxdv) = \EE \Big[  \sum_{u \in \partial \Tt_t} f(\xi_u^t, \tau_u) \Big]
$$
where $(\xi_u^t, \tau_u)$ is the size at time $t$ of $u$ together with its growth rate, constant through time. 
Then $n(t,dxdv)$ satisfies in the weak sense of measures the partial differential equation of {\bf Model (S+V)}. We refer to Theorem 1 of Doumic {\it et al.} \cite{4_DHKR1} for such a result.

\subsubsection{Stochastic tools}

Relying on a stochastic procedure, we numerically evaluate $\lambda_{\gamma,\rho}$ for given division rates $\gamma$ and growth rate densities $\rho$.
Our estimates of $\lambda_{\gamma,\rho}$ are based on trees observed between time $0$ and time $T$. The theoretical framework legitimating our approach is thus a continuous time setting as in Bansaye {\it et al.} \cite{4_BDMT} or Cloez \cite{4_cloez}. \\

\paragraph{{\it Estimation of the Malthus parameter.}}
The value of the Malthus parameter is encoded in the time evolution of the tree.
To build an estimator one can exploit the asymptotic behaviour of empirical means at time $t$, as studied in Cloez \cite{4_cloez},
\begin{equation} \label{4_eq:convmeandelta} 
e^{- \lambda_{\gamma,\rho} t}  \sum_{u \in \partial \Tt_t} f(\xi_u^t, \tau_u) \approx  c_{\gamma,\rho}(f) W_{\gamma,\rho} ,
\end{equation}
when $t$ is large, in probability, where $c_{\gamma,\rho}(f) $ is a constant and $W_{\gamma,\rho}$ a non-degenerated random variable.
Now define the sum of all sizes of living cells at time $t$, 
\begin{equation*} \label{4_eq:biomass}
\mathfrak{M}_t = \sum_{u\in \partial \Tt_t} \xi_u^t = \sum_{u\in \partial \Tt_t} \xi_u e^{\tau_u (t-b_u)},
\end{equation*}
called biomass in the biological literature.
If we observe $(\xi_u^{T/2} , u\in \partial \Tt_{T/2})$ and $(\xi_u^T, u \in \partial \Tt_{T})$ we can define an estimator of the Malthus parameter as
\begin{equation} \label{4_eq:lambbiomass}
\widehat{\lambda}_T = \frac{2}{T} \ln\Big( \frac{\mathfrak{M}_T}{\mathfrak{M}_{T/2}} \Big).
\end{equation}
Approximation \eqref{4_eq:convmeandelta} ensures that this estimator is close to the true value for large $T$. 
Note that a large collection of estimators based on \eqref{4_eq:convmeandelta} can be built, choosing different test functions, and we discuss this in Section \ref{4_sec:appendix} (see Figure \ref{4_fig:biomassvsnb}).
 We see that the estimator \eqref{4_eq:lambbiomass} requires the observation of the whole population at two different large times, chosen for simplicity as $T/2$ and $T$. Thus in order to numerically approximate $\lambda_{B,\rho}$ by \eqref{4_eq:lambbiomass} for given parameters $B$ and $\rho$ 
 one need to simulate continuous time trees up to a large time $T$.\\

\paragraph{{\it Simulation of a continuous time tree up to time $T$.}} \label{4_simulationfulltree}
Let us be given a division rate $B$ and a density $\rho$ on $\Vv$.
To simulate a full tree up to a time $T$, we have to keep the birth and division dates. Set $\xi_\emptyset = x_\emptyset$ some given initial value and $b_\emptyset = 0$. For any cell $u^{-} \in \TT$, given its birth time $b_{u^-}$ and its size at birth $\xi_{u^-}$, we compute $(\zeta_{u^-},\tau_{u^-})$ and its division time $d_{u^-}$ in the following way,
\begin{itemize}
\item[{\bf 1)}] Draw its growth rate $\tau_{u^-}$ according to $\rho$.
\item[{\bf 2)}] Draw $\xi_u$ given $\xi_{u^-} = x$ according to the transition $\Pp_{\gamma}(x,y)dy$.
\item[{\bf 3)}] Recalling \eqref{4_eq:2equalparts}, compute its lifetime $\zeta_{u^-}$ by 
\begin{equation} \label{eq:computezeta}
\zeta_{u^-} = \frac{1}{\tau_{u^-}} \ln(\frac{2 \xi_u}{\xi_{u^-}}).
\end{equation}
Set its division time $d_{u^-} = b_{u^-}+\zeta_{u^-}$ (and set $b_u = d_{u^-}$). 
Then,
\begin{enumerate}
\item[(i)] If $d_{u^-} \geq T$, keep $(b_{u^-}, \zeta_{u^-}, \xi_{u^-},\tau_{u^-})$ but drop $(b_u, \xi_u)$ and do not simulate further the descendants of $u^-$.
\item[(ii)] If $d_{u^-} < T$, keep $(b_u, \xi_u)$ and go back to Step {\bf 1)}.
\end{enumerate}
\end{itemize}

\subsection{Influence of variability on the Malthus parameter}

\subsubsection{Numerical protocol}

We first specify the division rate $\gamma$ and the growth rate Markov kernel~$\rho$. We recall that the concentration $n(t,x,v)$ of cells of size $x \geq 0$ and growth rate $v\in \Vv$ at time $t\geq 0$ evolves as
\begin{equation*} \label{4_eq:edp_sizevar}
\left\{
\begin{array}{l}
\frac{\partial}{\partial t} n(t,x,v) + \frac{\partial}{\partial x}\big(g_x(x,v) n(t,x,v)\big) + \gamma(x,v)n(t,x,v) = 4 \int_\Vv \gamma(2x,v') n(t,2x,v') \rho(v',v) dv', \\ \\
g_x(x=0,v') n(t,x=0, v') = 0, \quad \quad
n(t=0,x,v)=  n^{{\rm in}}(x,v).
\end{array}
\right.
\end{equation*}

\paragraph{{\it Division rate.}}
Mimicking again S. Taheri-Araghi {\it et al.} \cite{4_VergassolaJun}, we work with a division rate per unit of time $\gamma$ chosen as
\begin{equation} \label{4_eq:vxBdelta}
\gamma(x,v) = vxB(x)
\end{equation}
for a continuous $B: [0,\infty) \rightarrow [0,\infty)$ such that $\int^{\infty} B(s)ds = \infty$.
Note that $vxB(x)dt = B(x) dx$ (since $dx = g_x(x,v) dt$ with $g_x(x,v) = vx$) and that is why we call $B$ a division rate per unit of size.
As previously, from now on we denote the Malthus parameter $\lambda_{\gamma,\rho}$ by $\lambda_{B,\rho}$ to emphasize we work under Specification \eqref{4_eq:vxBdelta}. 
We choose $B$ of the general form 
\begin{equation} \label{4_eq:Bsimu}
B(x) = B_{x_0,\beta}(x) = (x-x_0)^{\beta} {\bf 1}_{\{ x \geq x_0 \}}
\end{equation}
for $x_0\geq 0$ and $\b\geq0$. This power form in $\beta$ with some lag $x_0$ is inspired by experimental data. For instance we refer to \cite{4_DHKR2}, Supplementary Figures S1 and S2, where the division rate is estimated for the age-structured and size-structured models. Adjusting $\b$ and $x_0$ is an easy way to mimic qualitatively all these curves. The estimation is not accurate for large sizes or ages since observations are lacking and we opt not to truncate $B$ making it constant after some threshold but to let it grow to infinity.
For such a choice of $\gamma$ and $B$, one has to draw $\xi_u$ given $\xi_{u^-}$ according to 
\begin{equation} \label{4_eq:sizetransition}
 \Pp_\gamma(x,y)dy = \PP(\xi_u \in dy | \xi_{u^-} = x) = 2 B(2y) \exp \big(-\int_{x/2}^y 2 B(2s) ds \big) {\bf 1}_{y \geq x/2}dy,
\end{equation}
in Step {\bf 2)} of the previous algorithm\footnote{Note that for the choice \eqref{4_eq:Bsimu} of $B$, we can easily do it  inverting the cumulative distribution function.}.
In our numerical tests, we pick $x_0 = 1$ and $\b = 2$.\\

\paragraph{{\it Parametrisation of the variability.}}
We suppose here that the growth rate of the daughter cell is independent of its parent, by picking a variability kernel $\rho(v,dv')$ independent of $v$. 
Given a non-degenerated baseline density $\rho : \Vv \rightarrow [0,\infty)$ with mean $\bar v$, we define, for $\alpha \in (0,1]$,
\begin{equation} \label{4_eq:rhoalpha2}
\rho_\a(v) = \a^{-1}\rho\big(\a^{-1}(v- (1-\a) \bar v)\big), \quad v \in \Vv_\a = \a \Vv + (1 - \a) \bar v.
\end{equation}
Note that $\rho_\a$ has also mean $\bar v$ (we preserve the mean growth rate).
We set $\int_\Vv (v-\bar v)^2 \rho(v)dv = \eta^2$.
The coefficient of variation, denoted by CV indexed by the probability distribution considered, is defined as the quotient of its standard error and mean.
Then $CV_\rho =  \eta/\bar v$ and $CV_{\rho_\a} = \a CV_{\rho}$. 
Set $\varphi(x) = (2\pi)^{-1/2} e^{-x^2/2}$ the density of the standard gaussian and $\Phi(x)= \int_{-\infty}^x \varphi(y) dy$ its cumulative distribution function.
We pick for the baseline density $\rho$,
\begin{equation} \label{4_eq:rho}
\rho(v) = \frac{\varphi\big(\sigma_\eta^{-1}(v-\bar v)\big) {\bf 1}_{v_{\min} \leq v \leq v_{\max}} }{\Phi\big(\sigma_\eta^{-1}(v_{\max}-\bar v)\big)-\Phi\big(\sigma_\eta^{-1}(v_{\min}-\bar v)\big)} 
\end{equation}
with $\bar v = (v_{\min}+v_{\max})/2$. That is to say, $\rho$ is the truncation on $ [v_{\min},v_{\max}] = \Vv$ of a Gaussian distribution with mean $\bar v$ and standard deviation $\sigma_\eta$.
We set $v_{\min} = 0$, $v_{\max} = 2$ so that $\bar v = 1$ and we choose $\sigma_\eta = 0.70$ so that $CV_\rho = 50\%$ (using the known formulae of the moments of a truncated Gaussian distribution). 
Once $\rho$ is fixed, it enables us to define the collection $\big(\rho_\a,\a\in(0,1]\big)$.  
For $\alpha \in (0,1]$ the support of the density $\rho_\a$ is $\Vv_\a = [1-\a,1+\a]$
and $CV_{\rho_\a} = \a/2$ for such a choice of~$\rho$. The larger the coefficient of variation $CV_{\rho_\a}$, the more variability in the growth rate. \\

\paragraph{{\it Estimation of the curve $CV_{\rho_\a} \leadsto \lambda_{B,\rho_\a}$.}}
For a given $\a \in (0,1)$, we simulate $M = 50$ continuous time trees up to a large time $T$, picking $x_\emptyset = 2$, $v_\emptyset = 1$, with division rate $B = B_{x_0 = 1,\beta = 2}$ and growth rate density
$$
\rho(v,dv')=\rho_\a(v')dv'.
$$ 
We obtain a collection 
$$
(\widehat{\lambda}_{T,m}^{(\a)} , m = 1,\ldots,M)
$$
of $M$ estimators of the Malthus parameter computed according to \eqref{4_eq:lambbiomass}.
We denote the mean of these $M$ estimators by $\widehat{\lambda}_{T}^{(\a)}$. It enable us to obtain a reconstruction of the curve $CV_{\rho_\a} \leadsto \lambda_{B,\rho_\a}$.\\

The question now states as follows: what can we say about $\lambda_{B,\rho_\a}$ compared to $\lambda_{B,\bar v} = \bar v$ the average growth rate, which is the Malthus parameter of a population without  variability in the growth rate? 

\subsubsection{Main result: the Malthus parameter decreases when introducing variability}
Figure~\ref{4_fig:SizeVar_lambdaalpha} below shows the curve
$$
CV_{\rho_\a} \leadsto \widehat{\lambda}_{T}^{(\a)},
$$
 together with a confidence interval for each $\a$, such that it contains 95\% of the $M$ estimators. The aim is to get confidence intervals of good precision and to that end we choose large enough times~$T$.
The mean numbers of living individuals at time $T$ is at least of magnitude $50~000$, see Supplementary Table~\ref{4_tab:size} in Section~\ref{4_sec:appendix}.  
(Such an amount of data, in a full tree case, would correspond to the observation of the first 16 generations.)
We first observe that $\lambda_{B,\rho_\a}$ is significantly lower than the value of reference $ \int_\Vv  v\rho_{\a}(v) dv = \bar v = 1$, for $CV_{\rho_\a}$ exceeding 10\%. In addition, the curve is significantly decreasing as $CV_{\rho_\a}$ increases, {\it i.e.} as there is more and more variability in the growth rate.
\begin{figure}[h!]
\begin{center}
\includegraphics[width=12cm]{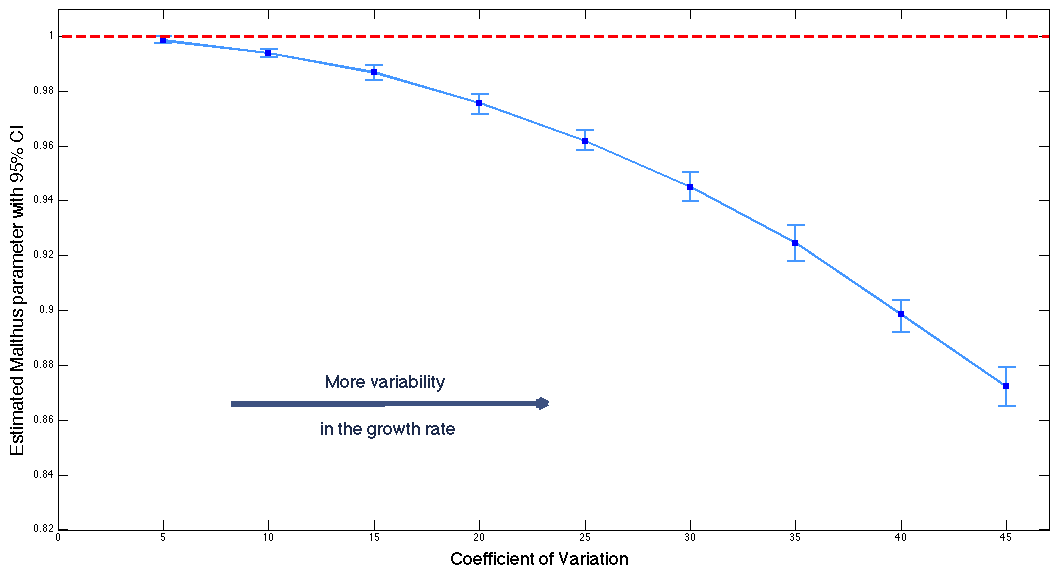}
\caption{{\it {\bf Model (S+V)}. Division rate $\gamma(x,v) = vxB(x)$ with $B(x) = (x-1)^2{\bf 1}_{\{x\geq1\}}$. Estimated curve $CV_{\rho_\a} \leadsto \lambda_{B,\rho_\a}$ using estimator \eqref{4_eq:lambbiomass} (mean and 95\% confidence interval based on $M=50$ Monte Carlo continuous time trees). Reference (all cells grow at a rate $\bar v=1$): $\lambda_{B,\bar v} = \bar v = 1$.} \label{4_fig:SizeVar_lambdaalpha}}
\end{center}
\end{figure}

\subsubsection{Robustness our results}

The result illustrated by Figure~\ref{4_fig:SizeVar_lambdaalpha} is robust when changing parameters or changing slightly the model. \\

\noindent {\it Division rate change.} The result is robust when changing the division rate $B$. In particular, we have explored several couples of parameters $(x_0,\b)$ (recall Parametrisation \eqref{4_eq:Bsimu}). The conclusion remains the same: the Malthus parameter significantly decreases when there is more and more variability in the growth rate.
Supplementary Table~\ref{4_tab:otherB} in Section~\ref{4_sec:appendix} displays the results for $x_0 = 1$ and $\beta = 8$.\\

\noindent {\it Asymmetric division.} So far we have assumed symmetry in the division. However a slight asymmetry can arise, see Marr {\it et al.} \cite{4_marr} or Soifer {\it et al.} \cite{4_Soifer} (Figure 8). 
Asymmetry is defined as the ratio of the size at birth of one daughter cell over the size at division of its parent. The distribution of this ratio shows a mode at $0.50$ and has a coefficient of variation of about $4\%$.
In our numerical study, we amplify asymmetry since we aim at measuring an effect of asymmetry on the curve $CV_{\rho_\a} \leadsto \lambda_{B,\rho_\a}$.
We immediately extend the model presented in Section \ref{subsec:size_sto} to allow for asymmetry: we assume that a cell of size $x$ divides into two cells of sizes $u x$ and $(1-u)x$ with $u$ uniformly drawn on the interval $[\ep,1-\ep]$, independently of everything else, for $0\leq \ep <1/2$. A dividing cell does not produce an arbitrarily small cell, that is why we chose $\ep \neq 0$. However we fix $\ep$ quite close to $0$ for our study. Supplementary Table~\ref{4_tab:asymm} in Section~\ref{4_sec:appendix} collects the results for $\ep = 0.10$. Once again there is no significant change compared to Figure~\ref{4_fig:SizeVar_lambdaalpha}.
Thus asymmetry seems to have no significant impact on the Malthus parameter, in presence of variability.\\

\noindent {\it Linear growth.} Consider now that the size at time $t$ of cell $u$ is $\xi_u^t = \xi_u + \tau_u t$ for $t\in[b_u,d_u)$. In order to simulate a continuous time tree, we still use the transition defined by \eqref{4_eq:sizetransition} in Step~{\bf 2)} of our algorithm and one has only to replace \eqref{eq:computezeta} by $\zeta_{u^-} = (2\xi_u-\xi_{u^-})/\tau_{u^-}$ in Step {\bf 3)}. 
We still observe a penalisation due to variability in terms of the overall cell population growth (see Supplementary Table \ref{4_tab:linear} in Section \ref{4_sec:appendix}). For coefficient of variations around 15\%-20\% in the growth rates, the decrease of the Malthus parameter is estimated at 1--1.5\%. \\

\noindent {\it Unit size versus unit time division rate.} Notice that the model studied in Doumic {\it et al.} \cite{4_DHKR1} corresponds to the choice $\gamma(x,v) = B(x)$ (there $B$ is a rate per unit of time) instead of $\gamma(x,v) = vxB(x)$ (here $B$ is a rate per unit of size). This is fundamentally different. As one can see in \eqref{4_eq:sizetransition}, with the choice $\gamma(x,v) = vxB(x)$, the size at birth of a cell actually does not depend on the growth rate of its parent whereas it would be the case with the choice $\gamma(x,v) = B(x)$, since
$$
\PP(\xi_{u^-} \in dy | \xi_u = x, \tau_u = v) = \frac{B(2y)}{vy} \exp\big(-\int_{x/2}^y \frac{B(2s)}{vs}ds\big) {\bf 1}_{\{y\geq x/2\}},
$$
(see \cite{4_DHKR1}, Equation (11)), and this is the main difference. In order to simulate a continuous time  tree up to a given time, at Step~{\bf 2)} of our algorithm, we draw $\xi_u$ given $\xi_{u^-}$  and $\tau_u$ (simulated in Step~{\bf 1)}) according to the previous equation (we use a rejection sampling algorithm for this Step~{\bf 2)}).
Whatever the specification is, our results concerning the Malthus parameter remain unchanged (see Supplementary Table \ref{4_tab:unitB} in Section \ref{4_sec:appendix}): we observe a decrease when there is more and more variability (the decrease seems slightly higher in the case $\gamma(x,v) = vxB(x)$ than in the case $\gamma(x,v) = B(x)$). \\

 Three main conclusions regarding {\bf Model (S+V)} are in order. For a unit size division rate $B$ experimentally plausible, when there is variability in the growth rate among cells, 
\begin{itemize}
\item[{\bf 1)}] The Malthus parameter is lower than the value of reference computed assuming all cells grow at the mean growth rate. 
\item[{\bf 2)}] The variation is of magnitude 2\% for experimentally realistic coefficients of variation in the growth rates distribution, around 15--20\%. 
\item[{\bf 3)}] In addition the Malthus parameter is monotonous: it decreases when there is more and more variability.
\end{itemize}
These conclusions are robust as argued above changing the unit size $B$, introducing asymmetry in the division, assuming the individual growth of each cell is linear instead of exponential or even considering a unit time division rate $B$. We stress that our conclusion {\bf 1)} coincides with conventional wisdom in biology and our methodology has the advantage of bringing some quantification through {\bf 2)}.

\section{Discussion}  \label{4_sec:discussion}

The scope of the perspectives is large and includes both theoretical, with analytical and statistical aspects, and experimental issues. We mention here some open questions. \\

\noindent {\it Eigenproblem.} In order to prove existence and uniqueness of the eigenelements in {\bf Model (A+V)}, we would like to find minimal assumptions on the aging speed $g_a$, on the division rate $\gamma$ and on the variability kernel $\rho$ (based on the general results of Mischler and Scher \cite{4_MischlerScher} for instance).  \\

\noindent {\it Without heredity -- $\rho(v,v') = \rho(v')$.} For {\bf Model (A+V)}, preserving the mean aging rate, can we build more general classes of $B$ (including experimentally more plausible $B$) in order to discriminate between the two cases $\lambda_{B,\rho}> \lambda_{B,\bar v}$ and $\lambda_{B,\rho} < \lambda_{B,\bar v}$?
For {\bf Model (S+V)}, preserving the mean individual growth rate, one can ask if the these two cases are possible. Is there some {\it plausible} $B$ such that $\lambda_{B,\rho} > \lambda_{B,\bar v}$?
We would like to build general classes of $B$ to discriminate between the two cases (specifying a density $\rho$ if needed, a truncated Gaussian for instance), and to compute the perturbations of $\lambda_{B,\rho_\a}$ around $\lambda_{B,\bar v}$ for $\rho_\a$ tending to $\delta_{\bar v}$ as $\a\rightarrow 0$ (using the same kind of tools as to prove Theorem~\ref{4_thm:ageperturbation}, see also Michel \cite{4_Michel06}). Still for {\bf Model (S+V)}, preserving the mean lifetime of the cells, which conclusions can be derived?\\

\noindent {\it With heredity -- general $\rho(v,v')$.}  We would like to take into account heredity in the transmission of the aging or growth rate, considering a general Markov kernel $\rho(v,v')$. 
In particular, for  {\bf Model (A+V)}, a natural question is: how does heredity in the transmission of the aging rate influence the results of Theorems~\ref{4_thm:ageinfluence} and~\ref{4_thm:ageperturbation}? Again, the choice of a quantity to preserve is crucial, and different possibilities appear. \\

\noindent {\it Alternative models.} Some other models successfully describe the division of {\it E. coli}. We refer to Amir \cite{4_ArielAmir} and Taheri-Araghi {\it et al.} \cite{4_VergassolaJun}. We wonder what can be said on the Malthus parameter in these two models. Preliminary answers are given in Olivier \cite{mathese} (Chapter 4).

\section{Proof of Theorem~\ref{4_thm:ageperturbation}} \label{4_sec:proofs}

As a preliminary, note that neglecting heredity in the transmission of the aging rate enables us to obtain explicit expressions for the eigenvectors $N_{\gamma,\rho}$ and $\phi_{\gamma,\rho}$, and an implicit relation which uniquely defines the Malthus parameter $\lambda_{\gamma,\rho}$.

\begin{lem} \label{4_lem:ageeigen}
Work under Assumption~\ref{4_ass:age_basic}. Assume in addition that $\rho(v,dv') = \rho(v') dv'$ for some continuous and bounded $\rho : \Vv \rightarrow [0,\infty)$ such that $\int_\Vv \rho(v')dv' = 1$.
Then the eigenvalue $\lambda_{\gamma,\rho}>0$ is uniquely defined by 
\begin{equation*}
2 \iint_{\Ss}  \frac{\gamma(a,v)}{g_a(a,v)} \exp \bigg( - \int_0^a \frac{\lambda_{\gamma,\rho} + \gamma(s,v)}{g_a(s,v)} ds \bigg) \rho(v) dv da = 1.
\end{equation*}
The unique solution $N_{\gamma,\rho}$ such that $(g_aN_{\gamma,\rho}) \in \Cc^1_b\big(\RR^+; \Cc_b(\Vv)\big)$ to \eqref{4_eq:EDP_N} and the unique solution $\phi_{\gamma,\rho} \in \Cc^1_b\big(\RR^+; \Cc_b(\Vv)\big)$ to \eqref{4_eq:EDP_Psi} are respectively given by, for any $(a,v)\in \Ss$,
\begin{equation*} \label{4_eq:age_N}
N_{\gamma,\rho}(a,v) =  \frac{\kappa \rho(v)}{g_a(a,v)} \exp \Big( - \int_0^a \frac{\lambda_{\gamma,\rho} + \gamma(s,v)}{g_a(s,v)} ds \Big), \quad \phi_{\gamma,\rho}(a,v) = \frac{\kappa' \int_a^{\infty} \gamma(s,v) N_{\gamma,\rho}(s,v) ds}{g_a(a,v) N_{\gamma,\rho}(a,v)}
\end{equation*}
with $\kappa$, $\kappa'$  normalizing constants such that $\iint_\Ss N_{\gamma,\rho}= 1$ and $\iint_\Ss N_{\gamma,\rho} \phi_{\gamma,\rho} =~1$.
\end{lem}

\noindent For a proof, one can easily check that $N_{\gamma,\rho}$ and $\phi_{\gamma,\rho}$ defined in such a way satisfy respectively \eqref{4_eq:EDP_N} and \eqref{4_eq:EDP_Psi}. The uniqueness is guaranteed by Theorem~\ref{4_thm:ageeig_general} (see a proof in the appendix, Section \ref{sec:proofvp}).\\

First note that since $\rho_\a$ converges in distribution to the Dirac mass at point $\bar v$ as $\a \rightarrow 0$, we get the convergence of $\lambda_{B,\rho_\a}$ to $\lambda_{B,\bar v}$ as $\a \rightarrow 0$ using the characterisations \eqref{4_eq:age_lambda} and \eqref{4_eq:age_lambda_novar}.
In a first step, we aim at computing the first derivative of $\a \leadsto \lambda_{B,\rho_\a}$.

\begin{prop}[First derivative] \label{4_prop:age_d1}
Consider {\bf Model (A+V)} with Specifications \eqref{4_eq:gv}, \eqref{4_eq:vBa} and $\rho(v,dv') = \rho_\a(v')dv'$ defined by \eqref{4_eq:rhoalpha},
\begin{multline*}
\frac{d\lambda_{B,\rho_\a}}{d\a} = \Big(\iint_{\Ss} \frac{a}{\a (v - \bar v) + \bar v}  \exp\big(\frac{- \lambda_{B,\rho_\a} a}{\a (v - \bar v) + \bar v}\big)  \Psi_B(a) \rho(v)  dvda \Big)^{-1} \\ \times
\iint_{\Ss}  \frac{ \lambda_{B,\rho_\a} a (v - \bar v) }{(\a (v - \bar v) + \bar v)^2} \exp\big( \frac{- \lambda_{B,\rho_\a} a }{\a (v - \bar v) + \bar v} \big) \Psi_B(a) \rho(v)  dvda
\end{multline*}
for any $\a \in (0,1]$.
\end{prop}

\begin{proof}[Proof]
Introduce the operator
\begin{align} \label{4_eq:adjointage}
\Aa_{\gamma,\rho} f(a,v) & = g_a(a,v) \frac{\partial}{\partial a} f(a,v) + \gamma(a,v) \Big( 2 \int_{\Vv} f(0,v') \rho(v') dv' - f(a,v) \Big)
\end{align}
densely defined on bounded continuous functions and let $\Aa_{\gamma,\rho}^*$ be its dual operator. The eigenproblem given by \eqref{4_eq:EDP_N} and \eqref{4_eq:EDP_Psi} can be written more shortly
$$
\Aa_{\gamma,\rho}^* N_{\gamma,\rho} = \lambda_{\gamma,\rho} N_{\gamma,\rho} , \quad \Aa_{\gamma,\rho} \phi_{\gamma,\rho} = \lambda_{\gamma,\rho} \phi_{\gamma,\rho} ,
$$
with $N_{\gamma,\rho}\geq 0$ such that $\iint_\Ss N_{\gamma,\rho} = 1$ and $\phi_{\gamma,\rho}\geq0$ such that $\iint_\Ss N_{\gamma,\rho}\phi_{\gamma,\rho} = 1$. In order to ease notation, when no confusion is possible, we abbreviate $\Aa_{\gamma,\rho_\a}$ by $\Aa_\a$, $N_{\gamma,\rho_\a}$ by $N_\a$, $\lambda_{\gamma,\rho_\a}$ by $\lambda_\a$ and so on.  We denote the support of $\rho_\a$ by $\Vv_\a$ and $[0,\infty)\times\Vv_\a$ by $\Ss_\a$.

\vip
\noindent {\it Step 1.}
Let $\a \in (0,1]$ be fixed. For $0<\ep<\a$, we claim that
\begin{equation} \label{4_eq:lemmeMichel}
 (\lambda_\a - \lambda_{\a-\ep} ) \iint_{\Ss_{\a-\ep}} \phi_{\a} N_{\a-\ep}  =   \iint_{\Ss_{\a-\ep}} (  \Aa_\a \phi_{\a} -  \Aa_{\a-\ep} \phi_{\a})  N_{\a-\ep}  .
\end{equation}
Indeed, operators $\Aa_\a$ and $\Aa^*_\a$ are dual, then
$$
\iint_{\Ss_{\a-\ep}} (  \Aa_\a \phi_{\a} -  \Aa_{\a-\ep} \phi_{\a}) N_{\a-\ep}  
 =  \iint_{\Ss_{\a-\ep}}   (\Aa_\a \phi_{\a})  N_{\a-\ep}  - \phi_{\a}   (\Aa_{\a-\ep}^* N_{\a-\ep} )
$$
which leads to \eqref{4_eq:lemmeMichel} since $\Aa_\a \phi_\a = \lambda_\a \phi_\a$  and $\Aa_{\a-\ep}^* N_{\a-\ep} = \lambda_{\a-\ep} N_{\a-\ep}$ (see also \cite{4_Michel06}, Lemma 3.2 and Equation (3.11)).
Using the definition \eqref{4_eq:adjointage} of  $\Aa_\a$, we get
\begin{equation*}
 \Aa_\a \phi_{\a}(a,v) - \Aa_{\a-\ep} \phi_{\a}(a,v) = 2 \gamma(a,v)  \Big( \int_{\Vv_{\a}} \phi_{\a} (0,v') \rho_{\a}(v') dv' - \int_{\Vv_{\a-\ep}} \phi_{\a} (0,v') \rho_{\a-\ep}(v') dv' \Big),
\end{equation*}
which we insert in \eqref{4_eq:lemmeMichel} to obtain
\begin{multline*}
\frac{\lambda_\a - \lambda_{\a-\ep}}{\ep} = \frac{\iint_{\Ss_{\a-\ep}} 2\gamma(a,v) N_{\a-\ep}(a,v) dv da}{ \iint_{\Ss_{\a-\ep}} \phi_\a(a,v) N_{\a-\ep}(a,v)dvda} \\ \times \Big( \frac{  \int_{\Vv_{\a}} f(0,v') \rho_{\a}(v') dv' - \int_{\Vv_{\a-\ep}} f(0,v') \rho_{\a-\ep}(v') dv' }{\ep} \Big)_{|f = \phi_\a}
\end{multline*}
 and we let $\ep$ go to zero to study the differentiability on the left of $\a \leadsto \lambda_\a$ at point $\a$. In the same way we compute $\ep^{-1}(\lambda_{\a+\ep} - \lambda_\a)$ to study the differentiability on the right of $\a \leadsto \lambda_\a$ at point $\a$. Since $N_{\a \pm \ep} \rightarrow N_{\a}$ pointwise when $\ep \rightarrow 0$ and $\iint_{\Ss_\a} \phi_\a N_\a = 1$, by Lebesgue dominated convergence theorem we get
\begin{equation} \label{4_eq:aged1formule1}
\frac{d\lambda_\a}{d\a}  = \kappa_\a \frac{d}{d \a} \Big( \int_{\Vv_{\a}}  f(0,v') \rho_{\a}(v') dv' \Big)_{\big| f = \phi_{\a}}
\end{equation}
with $\iint_{\Ss_\a} 2 \gamma(a,v) N_\a (a,v) dvda = \kappa_\a $ as defined in Lemma~\ref{4_lem:ageeigen}.
Thus, $\lambda_\a$ is differentiable in $\a$ if the derivative of the right-hand side exists.

\vip
\noindent {\it Step 2.} The adjoint eigenvector $\phi_\a$ has an explicit expression we exploit now.
For $\beta \in (0,1]$, recalling the explicit expression of $\phi_\beta$ given by Lemma~\ref{4_lem:ageeigen}, 
\begin{align*}
\phi_{\beta}(0,v) =  \kappa'_\beta \int_{0}^\infty \frac{\gamma(a,v)}{g_a(a,v)} \exp \big( - \int_0^a \frac{\lambda_{\beta} + \gamma(s,v)}{g_a(s,v)} ds \big) da, \quad v \in \Vv_\beta
\end{align*}
with $\kappa'_\beta$ defined in Lemma~\ref{4_lem:ageeigen}.
Thus \eqref{4_eq:aged1formule1} becomes
$$
\frac{d\lambda_\a}{d\a} = \kappa_\a \kappa'_\a \int_0^\infty \frac{d}{d\a} \Big( \int_{\Vv_\a} \frac{\gamma(a,v)}{g_a(a,v)} \exp \big( - \int_0^a \frac{\lambda_{\beta} + \gamma(s,v)}{g_a(s,v)} ds \big) \rho_\a(v) dv \Big)_{|\beta=\a}  da.
$$
Before going ahead in computations, recall Specifications \eqref{4_eq:gv}, $g_a(a,v) = v$, and  \eqref{4_eq:vBa}, $\gamma(a,v) = v B(a)$. The previous equality boils down to
\begin{equation} \label{4_eq:age_d1allrho}
\frac{d\lambda_\a}{d\a} = \bar \kappa_\a \int_0^\infty \frac{d}{d\a} \Big( \int_{\Vv_\a} \exp\big(\frac{- \lambda_{\beta}a}{v}\big) \rho_\a(v) dv \Big)_{|\beta=\a}  \Psi_B(a)  da
\end{equation}
with $\Psi_B$ defined by \eqref{4_eq:fB} and $\bar \kappa_\a = \kappa_\a \kappa'_\a$ equal to
\begin{equation} \label{4_eq:cste}
\bar \kappa_\a = \Big(\iint_{\Ss_\a} \frac{a}{v}  \exp(-\frac{\lambda_\a a}{v})  \Psi_B(a) \rho_\a(v)  dvda \Big)^{-1},
\end{equation}
using $\iint_{\Ss_\a} \phi_\a N_\a = 1$.

\vip
\noindent {\it Step 3.} 
In the case of the kernel \eqref{4_eq:rhoalpha}, we can explicitly compute the derivative with respect to $\a$. After a change of variables (setting $\a^{-1}(v - \bar v (1-\a))$ as new variable), \eqref{4_eq:age_d1allrho} reads
$$
\frac{d\lambda_\a}{d\a} = \bar \kappa_\a \int_{0}^\infty  \frac{d}{d \a} \Big( \int_{\Vv}  \exp\big( - \frac{\lambda_{\beta} a  }{\a (v - \bar v) + \bar v} \big) \rho(v) dv \Big)_{| \beta = \a} \Psi_B(a) da.
$$
Inverting the derivative in $\a$ and the integral and computing the derivative with respect to~$\a$, we get the announced result, recalling the definition \eqref{4_eq:cste} of $\bar \kappa_\a$.
\end{proof}

As a corollary of Proposition~\ref{4_prop:age_d1}, for all division rate $B$, when $\a$ goes to zero, the first derivative is null,
\begin{equation} \label{4_eq:d1nulle}
 \frac{d\lambda_{B,\rho_\a}}{d\a}\Big|_{\a = 0} = \lim_{\a \rightarrow 0} \frac{d\lambda_{B,\rho_\a}}{d\a} = 0
\end{equation}
since $\int_\Vv (v-\bar v) \rho(v) dv = 0$ (we picked a baseline density $\rho$ with mean $\bar v$). 
So we compute the second derivative when $\a$ converges to zero.

\begin{prop}[Second derivative at point $0$]  \label{4_prop:age_d2}
Consider {\bf Model (A+V)} with Specifications \eqref{4_eq:gv}, \eqref{4_eq:vBa} and $\rho(v,dv') = \rho_\a(v')dv'$ defined by \eqref{4_eq:rhoalpha},
$$
 \frac{d^2 \lambda_{B,\rho_\a}}{d\a^2} \Big|_{\a=0} = 
 \sigma^2
 \Big( \int_0^{\infty} \frac{a}{\bar v} e^{ - \frac{\lambda_{B,\bar v} a }{\bar v} } \Psi_B(a) da \Big)^{-1} 
 \int_0^\infty \frac{\lambda_{B,\bar v} a}{\bar v} \Big( \frac{\lambda_{B,\bar v} a}{\bar v} - 2  \Big) e^{- \frac{\lambda_{B,\bar v} a}{\bar v}} \Psi_B(a) da
$$
with $\sigma^2 = \int_\Vv (v - \bar v)^2 \rho(v) dv.$
\end{prop}

\begin{proof}[Proof]
For $\alpha \in (0,1]$, we set
$$
\Lambda_1(\a,a,u)  = \frac{a}{u} \exp\big( - \frac{\lambda_{B,\rho_\a} a }{u} \big), \quad \Lambda_2(\a,a,u)  = \frac{\lambda_{B,\rho_\a} a}{u^2} \exp\big( - \frac{\lambda_{B,\rho_\a} a }{u} \big),
$$
so that Proposition~\ref{4_prop:age_d1} reads
\begin{multline*}
\frac{d\lambda_{B,\rho_\a}}{d\a} = \Big( \iint_{\Ss} \Lambda_1(\a,a,\a(v-\bar v) + \bar v)  \Psi_B(a) \rho(v) dv da \Big)^{-1} \\ \times \iint_{\Ss} (v - \bar v) \Lambda_2(\a,a,\a(v-\bar v) + \bar v)  \Psi_B(a) \rho(v) dv da = D_1(\a)^{-1} D_2(\a)
\end{multline*}
say.
Then, $\frac{d^2\lambda_{B,\rho_\a}}{d\a^2}$ the second derivative of $\a \leadsto \lambda_{B,\rho_\a}$ can be written
\begin{multline*}
D_2(\a) D_1(\a)^{-2}  \iint_{\Ss}  \Big( \frac{\p \Lambda_1}{\p\a}(\a,a,\a(v-\bar v) + \bar v) + (v - \bar v)  \frac{\p \Lambda_1}{\p u}(\a, a,\a(v-\bar v) + \bar v) \Big) \Psi_B(a) \rho(v) dv da \\
+ D_1(\a)^{-1} \iint_{\Ss}  \Big( (v - \bar v) \frac{\p \Lambda_2}{\p\a}(\a,a,\a(v-\bar v) + \bar v) + (v - \bar v)^2  \frac{\p \Lambda_2}{\p u}(\a, a,\a(v-\bar v) + \bar v) \Big) \Psi_B(a) \rho(v) dv da.
\end{multline*}

We claim that the first term converges to zero as $\a \rightarrow 0$ and that 
\begin{equation} \label{4_dFdalpha}
\frac{\p \Lambda_2}{\p \a}(\a,a,\a(v-\bar v) + \bar v) \rightarrow 0,  \quad {\rm as} \;\; \a \rightarrow 0,
\end{equation}
so that,
\begin{equation} \label{4_eq: age_d2 interm}
\frac{d^2\lambda_{B,\rho_\a}}{d\a^2} \Big|_{\a=0}=  \big( \lim_{\a \rightarrow 0} D_1(\a) \big)^{-1} \iint_{\Ss} (v - \bar v)^2 \lim_{\a \rightarrow 0} \frac{\p \Lambda_2}{\p u}(\a, a,\bar v)  \Psi_B(a) \rho(v) dvda,
\end{equation}
since $\lim_{\a \rightarrow 0} \frac{\p \Lambda_2}{\partial u}(\a,a,\a(v - \bar v) + \bar v) = \lim_{\a \rightarrow 0} \frac{\p \Lambda_2}{\partial u}(\a,a,\bar v) $ by continuity of $u\leadsto \frac{\p \Lambda_2}{\partial u}(\a,a,u)$.\\

\noindent {\it Step 1}. 
We first treat the second term of $\frac{d^2\lambda_{B,\rho_\a}}{d\a^2}$, with three ingredients. 
{\bf 1)} 
In order to check \eqref{4_dFdalpha}, let us compute
$$
 \frac{\p \Lambda_2}{\p\a}(\a,a,u) = \Big( \frac{\frac{d\lambda_{B,\rho_\a}}{d\a} a}{u^2} - \frac{\lambda_{B,\rho_\a} a }{u^2} \frac{\frac{d\lambda_{B,\rho_\a}}{d\a} a}{u}  \Big)  e^{- \frac{\lambda_{B,\rho_\a} a}{u} }.
$$
As $\a \rightarrow 0$, $\lambda_{B,\rho_\a}$ converges to $\lambda_{B,\bar v}$ defined by \eqref{4_eq:age_lambda_novar}  and $\tfrac{d\lambda_{B,\rho_\a}}{d\a}$ to  $0$ (recall \eqref{4_eq:d1nulle}), thus $ \frac{\p \Lambda_2}{\p\a}(\a,a,u) \rightarrow 0$. Since $u \leadsto  \frac{\p \Lambda_2}{\p\a}(\a,a,u)$ is continuous, we deduce \eqref{4_dFdalpha}. 
{\bf 2)} Let us now compute
$$
\frac{\p \Lambda_2}{\p u}(\a, a,u) = \big( \frac{-2 \lambda_{B,\rho_\a} a }{u^3} + \frac{( \lambda_{B,\rho_\a} a )^2}{u^4} \big) e^{- \frac{\lambda_{B,\rho_\a} a}{u} },
$$
which leads to  
\begin{equation} \label{4_lim_dFdv}
\lim_{\a \rightarrow 0} \frac{\p \Lambda_2}{\p u}(\a, a,\bar v) = \frac{\lambda_{B,\bar v} a}{\bar v^3} \Big( \frac{\lambda_{B,\bar v} a}{\bar v} - 2  \Big)  e^{- \frac{\lambda_{B,\bar v} a}{\bar v}}
\end{equation}
since $\lambda_{B,\rho_\a}$ converges to $\lambda_{B,\bar v}$ defined by \eqref{4_eq:age_lambda_novar} as $\a \rightarrow 0$.
{\bf 3)} Since $ \Lambda_1(\a,a,\a(v-\bar v) + \bar v)$ converges to $\Lambda_1(0,a,\bar v)$ as $\a \rightarrow 0$, we get 
\begin{equation} \label{4_eq;lim_D1}
\lim_{\a \rightarrow 0} D_1(\a) =  \int_{0}^\infty \frac{a}{\bar v} \exp\big( \frac{- \lambda_{B,\bar v} a }{\bar v} \big) \Psi_B(a) da > 0
\end{equation}
using $\int_\Vv \rho(v)dv =1$.
Gathering \eqref{4_lim_dFdv} and \eqref{4_eq;lim_D1} enables us to compute the right-hand side of \eqref{4_eq: age_d2 interm}.

\vip
\noindent {\it Step 2}. We now check that the first term of $\frac{d^2\lambda_{B,\rho_\a}}{d\a^2}$ converges to zero as $\a \rightarrow 0$. One readily checks that 
$$
{\rm {\bf 1)}}\; \frac{\p \Lambda_1}{\p \a}(\a,a,\a(u'-\bar v) + \bar v) \rightarrow 0,  \quad {\rm as} \;\; \a \rightarrow 0, \quad {\rm {\bf 2)}} \; \lim_{\a \rightarrow 0} \frac{\p \Lambda_1}{\partial u}(\a,a,\a(v - \bar v) + \bar v)  < \infty.
$$
We know in addition that {\bf 3)} $\lim_{\a \rightarrow 0} D_1(\a) >0$ (see \eqref{4_eq;lim_D1}) and that {\bf 4)} $\lim_{\a \rightarrow 0} D_2(\a) = 0$ since $\int_\Vv (v-\bar v)\rho(v) dv = 0$. Gathering the four points enables us to conclude.
\end{proof}

\begin{proof}[Proof of Theorem~\ref{4_thm:ageperturbation}] 
For $\a\in[0,1)$,
$$
\lambda_{B,\rho_\a} = \lambda_{B,\bar v} + \a \frac{d\lambda_{B,\rho_\a}}{d\a}\Big|_{\a = 0}  + \frac{\a^2}{2} \frac{d^2 \lambda_{B,\rho_\a}}{d\a^2}\Big|_{\a = 0}  + o(\a^2)
$$
and we use Proposition~\ref{4_prop:age_d1} (or more precisely \eqref{4_eq:d1nulle}) and Proposition~\ref{4_prop:age_d2} to get the final result.
\end{proof}

\section{Supplementary figures and tables} \label{4_sec:appendix}

In the following tables, the variability kernel is defined by \eqref{4_eq:rhoalpha2} and \eqref{4_eq:rho}.\\

\noindent {\it Supplementary to Figure \ref{4_fig:SizeVar_lambdaalpha}.}
{\small
\begin{table}[h!]
\centering
\begin{tabular}{ccccc}
\hline \hline
 &$\boldsymbol{CV_{\rho_\a} = 5\%}$                                     & $\boldsymbol{CV_{\rho_\a} = 10\%}$                                     & $\boldsymbol{CV_{\rho_\a} = 15\%}$                                        & $\boldsymbol{CV_{\rho_\a} = 20\%}$                                          \\
$\boldsymbol{T}$                   & {\it 10.5}                                             & {\it 11}                                               & {\it 11.25}                                             & {\it 11.5}                                              \\
$\boldsymbol{\underset{{\rm (Min.}\leq\cdot\leq {\rm Max.)}}{{\rm Mean}}}\boldsymbol{|\partial \Tt_T|}$                 & $\underset{(42~358\leq\cdot\leq52~147)}{46~837}$ & $\underset{(57~254\leq\cdot\leq87~282)}{73~100}$ & $\underset{(53~615\leq\cdot\leq116~052)}{90~027}$ & $\underset{(68~946\leq\cdot\leq128~379)}{98~270}$  \\
$\boldsymbol{\underset{{\rm (sd.)}}{{\rm Mean}}} \, \boldsymbol{\widehat \lambda_T}$            & $\underset{(0.0006)}{0.9985}$                    & $\underset{(0.0009)}{0.9938}$                    & $\underset{(0.0014)}{0.9867}$                     & $\underset{(0.0018)}{0.9757}$                     \\
{\bf 95\% CI}         & $[0.9974,0.9999]$                                & $[0.9923,0.9954]$                                & $[0.9841,0.9893]$                                 & $[0.9717,0.9789]$                                                     \\ \hline \hline  

$\boldsymbol{CV_{\rho_\a} = 25\%}$                                         & $\boldsymbol{CV_{\rho_\a} = 30\%}$                                                & $\boldsymbol{CV_{\rho_\a} = 35\%}$                                            & $\boldsymbol{CV_{\rho_\a} = 40\%}$                                             & $\boldsymbol{CV_{\rho_\a} = 45\%}$                                             \\
{\it 11.75}                                              & {\it 12}                                                 & {\it 12.25}                                              & {\it 12.5}                                               & {\it 13}                                               \\
$\underset{(71~884\leq\cdot\leq157~032)}{107~305}$ & $\underset{(63~409\leq\cdot\leq200~860)}{120~102}$ & $\underset{(52~116\leq\cdot\leq172~328)}{104~628}$ & $\underset{(28~171\leq\cdot\leq192~021)}{117~208}$ &  $\underset{(39~238\leq\cdot\leq238~181)}{114~180}$ \\
 $\underset{(0.0019)}{0.9617}$                      & $\underset{(0.0027)}{0.9450}$                      & $\underset{(0.0035)}{0.9245}$                      & $\underset{(0.0030)}{0.8985}$                      & $\underset{(0.0039)}{0.8722}$                 \\
 $[0.9583,0.9656]$                                  & $[0.9397,0.9505]$                                  & $[0.9178,0.9312]$                                  & $[0.8920,0.9036]$                                  & $[0.8650,0.8794]$           \\ \hline \hline  
 
\end{tabular}
\caption{{\it {\bf Model (S+V).} Division rate $\gamma(x,v) = vxB(x)$ with $B(x) = (x-1)^2 {\bf 1}_{\{x \geq 1\}}$. Estimation of the Malthus parameter $\lambda_{B,\rho_\a}$ (mean and 95\% confidence interval based on $M=50$ Monte Carlo continuous time trees simulated up to time $T$) with respect to the coefficient of variation of the growth rates density $\rho_\a$ with mean $\bar v = 1$.
Reference (all cells grow at a rate $\bar v=1$): $\lambda_{B,\bar v} = 1$.
} \label{4_tab:size}}
\end{table}
}

\newpage

\noindent {\it Robustness of our results: division rate change.}
{\small
\begin{table}[h!]
\centering
\begin{tabular}{ccccc}
\hline \hline
&$\boldsymbol{CV_{\rho_\a} = 5\%}$                                     
& $\boldsymbol{CV_{\rho_\a} = 10\%}$                                     
& $\boldsymbol{CV_{\rho_\a} = 15\%}$                                        
& $\boldsymbol{CV_{\rho_\a} = 20\%}$		 				 	\\
$\boldsymbol{T}$                   
& {\it 10.5}                                             
& {\it 11}                                               
& {\it 11.25}                                             
& {\it 11.5}                                           			 				  	\\
$\boldsymbol{\underset{{\rm (Min.}\leq\cdot\leq {\rm Max.)}}{{\rm Mean}}}\boldsymbol{|\partial \Tt_T|}$               
& $\underset{(41~357 \leq \cdot \leq 53~270)}{47~160}$                        
& $\underset{(61~758 \leq \cdot \leq 84~191)}{73~670}$                        
& $\underset{(66~499 \leq \cdot \leq 118~486)}{86~410}$                         
& $\underset{(61~924 \leq \cdot \leq 127~299)}{95~230}$               				      		\\    
$\boldsymbol{\underset{{\rm (sd.)}}{{\rm Mean}}} \, \boldsymbol{\widehat \lambda_T}$           
& $\underset{(0.0005)}{0.9984}$ 
& $\underset{(0.0009)}{0.9934}$ 
& $\underset{(0.0012)}{0.9855}$  
& $\underset{(0.0015)}{0.9732}$ 		 									\\
{\bf 95\% CI}   
& $[ 0.9975 , 0.9995 ]$           
& $[ 0.9918 , 0.9952 ]$             
& $[ 0.9833 , 0.9876 ]$             
& $[ 0.9705 , 0.9763 ]$               											\\ \hline \hline
$\boldsymbol{CV_{\rho_\a} = 25\%}$                                         
& $\boldsymbol{CV_{\rho_\a} = 30\%}$                                                
& $\boldsymbol{CV_{\rho_\a} = 35\%}$                                            
& $\boldsymbol{CV_{\rho_\a} = 40\%}$                                             
& $\boldsymbol{CV_{\rho_\a} = 45\%}$                                             		\\
{\it 11.75}                                              
& {\it 12}                                                 
& {\it 12.25}                                              
& {\it 12.5}                                               
& {\it 13}                                           							    	\\
$\underset{(53~902 \leq \cdot \leq 156~868)}{104~480}$                    
& $\underset{(53~156 \leq \cdot \leq 145~125)}{107~540}$                        
& $\underset{(50~784 \leq \cdot \leq 162~615)}{100~480}$                         
& $\underset{(42~533 \leq \cdot \leq 192~984)}{90~440}$                      
& $\underset{(22~600 \leq \cdot \leq 200~034)}{102~880}$          				          	\\
$\underset{(0.0019)}{0.9589}$  
& $\underset{(0.0023)}{0.9384}$ 
& $\underset{(0.0025)}{0.9166}$  
& $\underset{(0.0036)}{0.8890}$  
& $\underset{(0.0044)}{0.8597}$  											\\
$[ 0.9554 , 0.9628 ]$             
& $[ 0.9332 , 0.9426 ]$             
& $[ 0.9113 , 0.9214 ]$             
& $[ 0.8820 , 0.8945 ]$            
& $[ 0.8489 , 0.8655 ]$         												\\ \hline \hline 
\end{tabular}

\caption{{\it {\bf Model (S+V).} Division rate $\gamma(x,v) = vxB(x)$ with $B(x) = (x-1)^8 {\bf 1}_{\{x \geq 1\}}$. Estimation of the Malthus parameter $\lambda_{B,\rho_\a}$ (mean and 95\% confidence interval based on $M=50$ Monte Carlo continuous time trees simulated up to time $T$) with respect to the coefficient of variation of the growth rates density $\rho_\a$ with mean $\bar v = 1$.
Reference (all cells grow at a rate $\bar v=1$):  $\lambda_{B,\bar v} = 1$.
} \label{4_tab:otherB}}
\end{table}
}


\noindent {\it Robustness of our results: asymmetric division.}

{\small
\begin{table}[h!]
\centering
\begin{tabular}{ccccc}
\hline \hline
&$\boldsymbol{CV_{\rho_\a} = 5\%}$                                     
& $\boldsymbol{CV_{\rho_\a} = 10\%}$                                     
& $\boldsymbol{CV_{\rho_\a} = 15\%}$                                        
& $\boldsymbol{CV_{\rho_\a} = 20\%}$		 				 	\\
$\boldsymbol{T}$                   
& {\it 10.5}                                             
& {\it 11}                                               
& {\it 11.25}                                             
& {\it 11.5}                                           			 				  	\\
$\boldsymbol{\underset{{\rm (Min.}\leq\cdot\leq {\rm Max.)}}{{\rm Mean}}}\boldsymbol{|\partial \Tt_T|}$   
& $\underset{(49~880 \leq \cdot \leq 59~486)}{53~590}$                        
& $\underset{(69~343 \leq \cdot \leq 97~237)}{85~310}$                        
& $\underset{(82~182 \leq \cdot \leq 129~410)}{101~350}$                         
& $\underset{(86~751 \leq \cdot \leq 154~226)}{121~570}$               				      		\\
$\boldsymbol{\underset{{\rm (sd.)}}{{\rm Mean}}} \, \boldsymbol{\widehat \lambda_T}$   
& $\underset{(0.0006)}{0.9987}$ 
& $\underset{(0.0008)}{0.9948}$ 
& $\underset{(0.0014)}{0.9880}$  
& $\underset{(0.0016)}{0.9783}$ 		 									\\
{\bf 95\% CI}     
& $[ 0.9972 , 0.9996 ]$           
& $[ 0.9932 , 0.9963 ]$             
& $[ 0.9855 , 0.9906 ]$             
& $[ 0.9755 , 0.9824 ]$               											\\ \hline \hline
$\boldsymbol{CV_{\rho_\a} = 25\%}$                                         
& $\boldsymbol{CV_{\rho_\a} = 30\%}$                                                
& $\boldsymbol{CV_{\rho_\a} = 35\%}$                                            
& $\boldsymbol{CV_{\rho_\a} = 40\%}$                                             
& $\boldsymbol{CV_{\rho_\a} = 45\%}$                                             		\\
{\it 11.75}                                              
& {\it 12}                                                 
& {\it 12.25}                                              
& {\it 12.5}                                               
& {\it 13}                                           							    	\\
$\underset{(84~620 \leq \cdot \leq 234~613)}{129~770}$                    
& $\underset{(67~334 \leq \cdot \leq 222~004)}{135~650}$                        
& $\underset{(50~493 \leq \cdot \leq 234~646)}{141~660}$                         
& $\underset{(23~530 \leq \cdot \leq 243~023)}{140~170}$                      
& $\underset{(18~187 \leq \cdot \leq 359~824)}{154~120}$          				          	\\
$\underset{(0.0019)}{0.9665}$  
& $\underset{(0.0021)}{0.9511}$ 
& $\underset{(0.0026)}{0.9322}$  
& $\underset{(0.0038)}{0.9099}$  
& $\underset{(0.0039)}{0.8836}$  											\\
$[ 0.9634 , 0.9706 ]$             
& $[ 0.9472 , 0.9545 ]$             
& $[ 0.9263 , 0.9372 ]$             
& $[ 0.9018 , 0.9166 ]$            
& $[ 0.8743 , 0.8925 ]$         												\\ \hline \hline 
\end{tabular}

\caption{{\it {\bf Model (S+V).} Division rate $\gamma(x,v) = vxB(x)$ with $B(x) = (x-1)^2 {\bf 1}_{\{x \geq 1\}}$. Asymmetric division (a cell of size $x$ splits into two cells of size $ux$ and $(1-u)x$ for $u$ uniformly drawn on $[0.1,0.9]$). Estimation of the Malthus parameter $\lambda_{B,\rho_\a}$ (mean and 95\% confidence interval based on $M=50$ Monte Carlo continuous time trees simulated up to time $T $) with respect to the coefficient of variation of the growth rates density $\rho_\a$ with mean $\bar v = 1$.
Reference (all cells grow at a rate $\bar v=1$): $\lambda_{B,\bar v} = 1$.
} \label{4_tab:asymm}}
\end{table}
}

\newpage

\noindent {\it Robustness of our results: linear growth.}
{\small
\begin{table}[h!]
\centering
\begin{tabular}{ccccc}
\hline \hline
&$\boldsymbol{CV_{\rho_\a} = 5\%}$                                     
& $\boldsymbol{CV_{\rho_\a} = 10\%}$                                     
& $\boldsymbol{CV_{\rho_\a} = 15\%}$                                        
& $\boldsymbol{CV_{\rho_\a} = 20\%}$		 				 	\\
$\boldsymbol{T}$                   
& {\it 17.5}                                             
& {\it 18}                                               
& {\it 18.25}                                             
& {\it 18.5}                                           			 				  	\\
$\boldsymbol{\underset{{\rm (Min.}\leq\cdot\leq {\rm Max.)}}{{\rm Mean}}}\boldsymbol{|\partial \Tt_T|}$         
& $\underset{(38~138 \leq \cdot \leq 71~168)}{55~219}$                        
& $\underset{(43~802 \leq \cdot \leq 95~071)}{67~760}$                        
& $\underset{(40~296 \leq \cdot \leq 113~904)}{75~748}$                         
& $\underset{(32~035 \leq \cdot \leq 119~198)}{76~084}$               				      		\\
$\boldsymbol{\underset{{\rm (sd.)}}{{\rm Mean}}} \, \boldsymbol{\widehat \lambda_T}$           
& $\underset{(0.0014)}{0.6116}$ 
& $\underset{(0.0014)}{0.6090}$ 
& $\underset{(0.0015)}{0.6043}$  
& $\underset{(0.0018)}{0.5976}$ 		 									\\
{\bf 95\% CI}       
& $[ 0.6086 , 0.6138 ]$           
& $[ 0.6066 , 0.6115 ]$             
& $[ 0.6017 , 0.6071 ]$             
& $[ 0.5945 , 0.6010 ]$               											\\ \hline \hline

$\boldsymbol{CV_{\rho_\a} = 25\%}$                                         
& $\boldsymbol{CV_{\rho_\a} = 30\%}$                                                
& $\boldsymbol{CV_{\rho_\a} = 35\%}$                                            
& $\boldsymbol{CV_{\rho_\a} = 40\%}$                                             
& $\boldsymbol{CV_{\rho_\a} = 45\%}$                                             		\\
{\it 18.75}                                              
& {\it 19}                                                 
& {\it 19.25}                                              
& {\it 19.5}                                               
& {\it 20}                                           							    	\\
$\underset{(29~071 \leq \cdot \leq 131~343)}{73~931}$                    
& $\underset{(30~940 \leq \cdot \leq 141~046)}{76~074}$                        
& $\underset{(28~704 \leq \cdot \leq 118~295)}{69~719}$                         
& $\underset{(10~488 \leq \cdot \leq 120~506)}{57~913}$                      
& $\underset{(3~017 \leq \cdot \leq 190~355)}{62~582}$          				          	\\
$\underset{(0.0021)}{0.5893}$  
& $\underset{(0.0023)}{0.5788}$ 
& $\underset{(0.0025)}{0.5658}$  
& $\underset{(0.0033)}{0.5513}$  
& $\underset{(0.0038)}{0.5348}$  											\\
$[ 0.5838 , 0.5942 ]$             
& $[ 0.5752 , 0.5861 ]$             
& $[ 0.5607 , 0.5702 ]$             
& $[ 0.5438 , 0.5578 ]$            
& $[ 0.5270 , 0.5413 ]$         												\\ \hline \hline 
\end{tabular}

\caption{{\it {\bf Model (S+V).} Division rate $\gamma(x,v) = v B(x)$ with $B(x) = (x-1)^2 {\bf 1}_{\{x \geq 1\}}$. Estimation of the Malthus parameter $\lambda_{B,\rho_\a}$ (mean and 95\% confidence interval based on $M=50$ Monte Carlo continuous time trees simulated up to time $T$) with respect to the coefficient of variation of the growth rates density $\rho_\a$ with mean $\bar v = 1$. 
Reference (all cells grow at a rate $\bar v=1$): $\lambda_{B,\bar v} \approx 0.6130$ {\footnotesize (over 50 continuous time trees simulated up to time 17.25, sd. 0.0016). Among the 50 realisations, 95\% lie between $0.6098$ and $0.6161$. The mean-size of the 50 trees is 46~353 (the smallest tree counts $30~553$ cells and the largest $70~914$).}
} \label{4_tab:linear}}
\end{table}
}


\noindent {\it Robustness of our results: unit size versus unit time division rate.}
{\small
\begin{table}[h!]
\centering
\begin{tabular}{ccccc}
\hline \hline
&$\boldsymbol{CV_{\rho_\a} = 5\%}$                                     
& $\boldsymbol{CV_{\rho_\a} = 10\%}$                                     
& $\boldsymbol{CV_{\rho_\a} = 15\%}$                                        
& $\boldsymbol{CV_{\rho_\a} = 20\%}$		 				 	\\
$\boldsymbol{T}$                   
& {\it 10.5}                                             
& {\it 10.75}                                               
& {\it 11}                                             
& {\it 11.25}                                           			 				  	\\
$\boldsymbol{\underset{{\rm (Min.}\leq\cdot\leq {\rm Max.)}}{{\rm Mean}}}\boldsymbol{|\partial \Tt_T|}$    
& $\underset{(35~256 \leq \cdot \leq 42~659)}{39~660}$                        
& $\underset{(38~675 \leq \cdot \leq 60~374)}{49~520}$                        
& $\underset{(47~384 \leq \cdot \leq 82~371)}{61~150}$                         
& $\underset{(48~639 \leq \cdot \leq 111~048)}{79~470}$               				      		\\
$\boldsymbol{\underset{{\rm (sd.)}}{{\rm Mean}}} \, \boldsymbol{\widehat \lambda_T}$           
& $\underset{(0.0006)}{0.9993}$ 
& $\underset{(0.0013)}{0.9974}$ 
& $\underset{(0.0016)}{0.9937}$  
& $\underset{(0.0019)}{0.9893}$ 		 									\\
{\bf 95\% CI}         
& $[ 0.9982 , 1.0006 ]$           
& $[ 0.9949 , 0.9996 ]$             
& $[ 0.9894 , 0.9966 ]$             
& $[ 0.9861 , 0.9933 ]$               											\\ \hline \hline
$\boldsymbol{CV_{\rho_\a} = 25\%}$                                         
& $\boldsymbol{CV_{\rho_\a} = 30\%}$                                                
& $\boldsymbol{CV_{\rho_\a} = 35\%}$                                            
& $\boldsymbol{CV_{\rho_\a} = 40\%}$                                             
& $\boldsymbol{CV_{\rho_\a} = 45\%}$                                             		\\
{\it 11.5}                                              
& {\it 11.75}                                                 
& {\it 12}                                              
& {\it 12.25}                                               
& {\it 12.5}                                           							    	\\
$\underset{(50~665 \leq \cdot \leq 139~785)}{92~490}$                    
& $\underset{(60~083 \leq \cdot \leq 171~387)}{109~600}$                        
& $\underset{(48~810 \leq \cdot \leq 231~667)}{124~320}$                         
& $\underset{(45~032 \leq \cdot \leq 239~816)}{143~760}$                      
& $\underset{(43~934 \leq \cdot \leq 287~633)}{146~400}$          				          	\\
$\underset{(0.0022)}{0.9827}$  
& $\underset{(0.0027)}{0.9743}$ 
& $\underset{(0.0034)}{0.9644}$  
& $\underset{(0.0030)}{0.9530}$  
& $\underset{(0.0041)}{0.9400}$  											\\
$[ 0.9784 , 0.9878 ]$             
& $[ 0.9688 , 0.9794 ]$             
& $[ 0.9588 , 0.9715 ]$             
& $[ 0.9466 , 0.9590 ]$            
& $[ 0.9317 , 0.9470 ]$         												\\ \hline \hline 
\end{tabular}

\caption{{\it {\bf Model (S+V).} Division rate $\gamma(x,v) = B(x)$ with $B(x) = (x-1)^2 {\bf 1}_{\{x \geq 1\}}$. Estimation of the Malthus parameter $\lambda_{B,\rho_\a}$ (mean and 95\% confidence interval based on $M=50$ Monte Carlo continuous time trees simulated up to time $T$) with respect to the coefficient of variation of the growth rates density $\rho_\a$ with mean $\bar v = 1$.
Reference (all cells grow at a rate $\bar v=1$): $\lambda_{B,\bar v} = 1$.
} \label{4_tab:unitB}}
\end{table}
}


\noindent {\it Number of cells versus biomass.}
Recall Approximation \eqref{4_eq:convmeandelta}.
Observing $\big( (\xi_u^t , \tau_u) , u \in \partial \Tt_t \big)$, or only a component of it for all living cells, at two different times, $T/2$ and $T$ for instance, with $T$ large enough, one can estimate the Malthus parameter, with a free choice for the smooth test function~$f$. 

\begin{figure}[h!]
\begin{center}
\includegraphics[width=7cm]{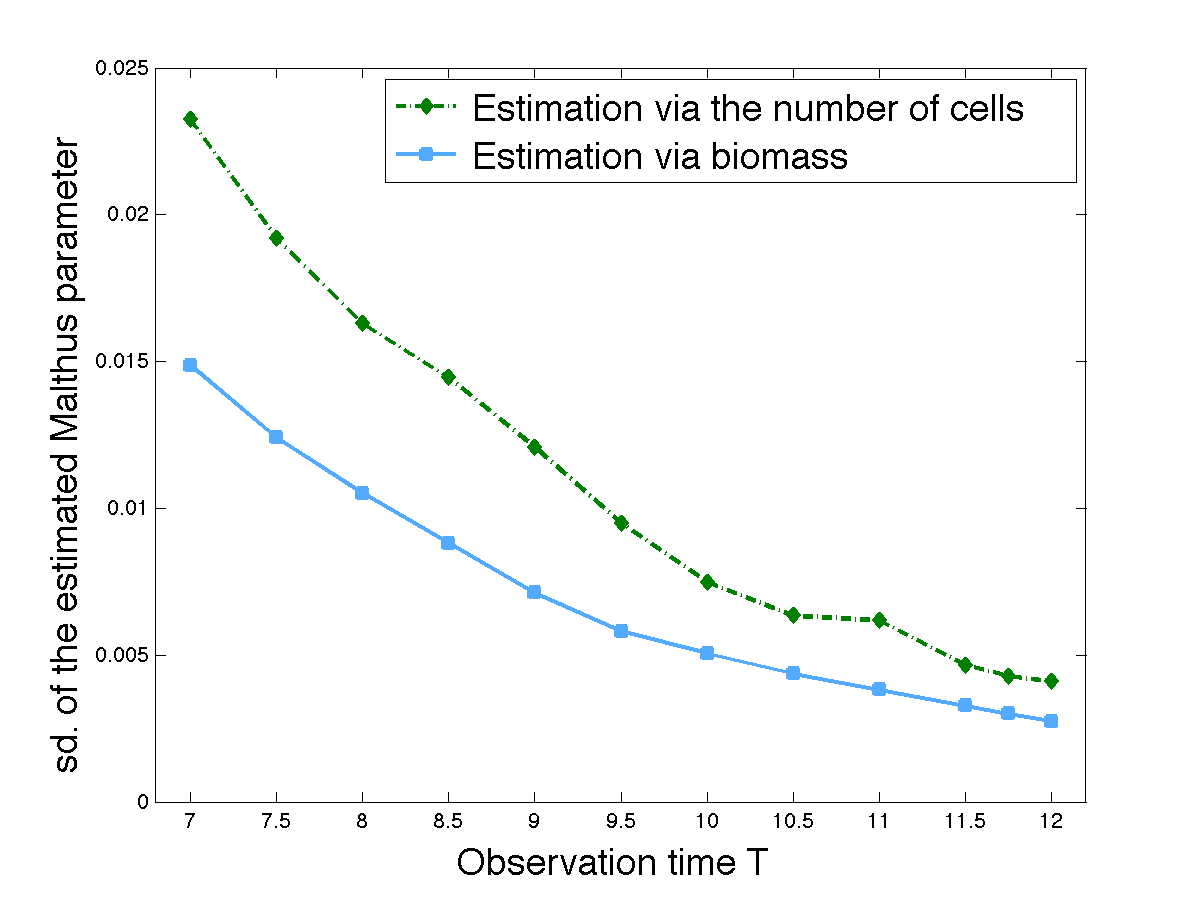}
\caption{{\it  {\bf Model (S+V).} Standard deviation of two estimators of the Malthus parameter as $T$ increases (based on $M=50$ Monte Carlo continuous time trees simulated up to time $T$), for $\rho_{\a = 0.3}$ and division rate $\gamma(x,v) = vxB(x)$ with $B(x) = (x-1)^2 {\bf 1}_{\{x\geq 1\}}$. Blue lower curve: estimation by \eqref{4_eq:lambbiomass} via the biomass. Green upper curve: estimation by \eqref{4_eq:lambnb} via the number of cells.} \label{4_fig:biomassvsnb}}
\end{center}
\end{figure}

One common choice is $f(x,v) = x$ (the empirical mean defines in this case the mean size of the cells) and it lead us to the estimator \eqref{4_eq:lambbiomass}.  
Another choice would be $f \equiv 1$ (the empirical mean defining here the mean number of cells) and it lead us to the estimator
\begin{equation} \label{4_eq:lambnb}
\widehat{\lambda}_T = \frac{2}{T} \ln\Big( \frac{|\partial \Tt_T|}{|\partial \Tt_{T/2}|} \Big)
\end{equation}
where $|\partial \Tt_T|$ stands for the cardinality of the set \eqref{4_eq:livingset} of living particles at time $T$.
It is interesting to compare these two estimators. 
For a fixed~$T$, the estimator \eqref{4_eq:lambbiomass} is better than \eqref{4_eq:lambnb} in the sense that its standard deviation is smaller, as pointed out in Supplementary Figure~\ref{4_fig:biomassvsnb}. 
However  the two estimators perform equivalently for large $T$.

\section{Appendix: Proof of Theorem~\ref{4_thm:ageeig_general}} \label{sec:proofvp}

Theorem~\ref{4_thm:ageeig_general} concerns the eigenproblem of the age-structured model with variability.
Contrary to Lemma \ref{4_lem:ageeigen}, we now work in the general case where $\rho$ is a Markov kernel.

\begin{proof}[Proof of Theorem~\ref{4_thm:ageeig_general}]
We first study the direct eigenproblem \eqref{4_eq:EDP_N}, then we turn to the adjoint eigenproblem \eqref{4_eq:EDP_Psi}. At last we prove uniqueness of the eigenelements.
The methodology we use is inspired by the one of \cite{4_Doumic07}.
In this proof we denote by  $\Cc_b(\Vv) = \Xx$ the Banach space of bounded continuous functions $f : \Vv \rightarrow \RR$, equipped with the supremum norm, $\| f \|_\Xx = \sup_{v\in\Vv} |f(v)|$, for $\Vv$ compact set of $(0,\infty)$.
\begin{proof}[Direct eigenproblem]
We split the proof into four steps.
\vip
\noindent {\it Step 1.} Since $N_{\gamma,\rho}$ satisfies \eqref{4_eq:EDP_N}, we immediately deduce that, for any $(a,v)\in \Ss$,
\begin{equation} \label{4_eq:eigenvect}
(g_aN_{\gamma,\rho})(a,v) = (g N_{\gamma,\rho})(0,v) \exp\Big( - \int_0^a \frac{\lambda_{\gamma,\rho} + \gamma(s,v)}{g_a(s,v)} ds \Big).
\end{equation}
Using the boundary condition leads us to
\begin{equation} \label{4_eq:lambdasolof}
(g_aN_{\gamma,\rho})(0,v') = 2 \iint_{\Ss} (g N_{\gamma,\rho})(0,v) e^{- \int_0^a \frac{\lambda_{\gamma,\rho}}{g_a(s,v)} ds} \Psi_{\gamma/g_a}(a,v) \rho(v,v') dv da
\end{equation}
for $v'\in\Vv$, setting
\begin{equation} \label{4_eq:fgammasurg}
\Psi_{\gamma/g_a}(a,v) = \frac{\gamma(a,v) }{g_a(a,v)}  \exp\Big( - \int_0^a \frac{\gamma(s,v)}{g_a(s,v)} ds \Big), \quad (a,v)\in\Ss.
\end{equation}
Note that, for every $v\in\Vv$, $a\leadsto \Psi_{\gamma/g_a}(a,v)$ is a density.
Equation \eqref{4_eq:lambdasolof} leads us to define an operator $\Gg_{\lambda} : \Xx \rightarrow \Xx$ by
\begin{equation*} \label{4_eq:Glambda}
(\Gg_{\lambda} f)(v') = 2 \iint_{\Ss} f(v) e^{- \int_0^a \frac{\lambda}{g_a(s,v)} ds }  \Psi_{\gamma/g_a}(a,v) \rho(v,v') dv da, \quad v' \in \Vv,
\end{equation*}
for any $\lambda \geq 0$.
In the next steps, we look for a solution $(\lambda,f)$ to the equation $\Gg_{\lambda} (f) = f$.
We work on the space $\Xx$ of continuous functions since we want to apply the Krein-Rutman theorem  \cite{4_DautrayLions} 
(then the interior of the positive cone is the set of positive functions). 
\vip
\noindent {\it Step 2.} 

We introduce a so-called regularized operator, which is strictly positive.
For a fixed $\ep > 0$, set
\begin{equation} \label{4_eq:rhoep}
\rho_\ep(v,v') = \rho(v,v') + |\Vv|^{-1} \ep , \quad (v,v')\in \Vv^2,
\end{equation}
where $|\Vv|$ stands for the Lebesgue measure of the compact set $\Vv\subset (0,\infty)$ and define the operator $\Gg_{\lambda,\ep} : \Xx \rightarrow \Xx$ by
\begin{equation*} \label{4_eq:Glambda}
(\Gg_{\lambda,\ep} f)(v') = 2 \iint_{\Ss} f(v) \exp\Big( - \int_0^a \frac{\lambda}{g_a(s,v)} ds\Big)  \Psi_{\gamma/g_a}(a,v) \rho_\ep(v,v') dv da, \quad v' \in \Vv,
\end{equation*}
for any $\lambda \geq 0$.
We claim that $\Gg_{\lambda,\ep}$ is {\bf 1)} strictly positive on $\Xx$ (\textit{i.e.} for any $f\in \Xx$ non-negative and different from the null function, $(\Gg_{\lambda,\ep}f)(v')>0$ for any $v' \in\Vv$), {\bf 2)} a linear mapping from $\Xx$ into itself, {\bf 3)} continuous and {\bf 4)}  compact.
Thus we are now in position to apply the Krein-Rutman theorem (we use Theorem 6.5 of \cite{4_Perthame}). For any $\lambda\geq0$ there exist a unique $\mu_{\lambda,\ep}>0$ and a unique positive $U_{\lambda,\ep}\in \Xx$ such that
\begin{equation} \label{4_eq:mulambdaep}
\Gg_{\lambda,\ep} (U_{\lambda,\ep}) = \mu_{\lambda,\ep} U_{\lambda,\ep}
\end{equation}
and $\|U_{\lambda,\ep}\|_\Xx = 1$.
\vip
It just remains to prove the four claimed properties. {\bf 1)} is precisely achieved thanks to the regularisation $\rho_\ep$ of $\rho$ by \eqref{4_eq:rhoep}. {\bf 2)} The linearity is obvious and for $f\in \Xx$, we have $\Gg_{\lambda,\ep} (f) \in \Xx$ since $v' \leadsto \rho_\ep(v,v')$ is continuous and bounded for any $v\in \Vv$. {\bf 3)} We even achieve Lipschitz continuity, for any $(f,g)\in\Xx^2$,
$$
\|\Gg_{\lambda,\ep} (f)-\Gg_{\lambda,\ep} (g)\|_\Xx \leq 2 (|\Vv| |\rho|_\infty +\ep) \|f-g\|_\Xx
$$
where $|\rho|_\infty = \sup_{(v,v')\in\Vv^2} \rho(v,v')$.
 {\bf 4)}
 We prove that for any $\lambda\geq 0$ the family $\big( \Gg_{\lambda,\ep} (f) , f\in \Xx \big)$ is equicontinuous. 
Indeed $(\Gg_{\lambda,\ep}f)(v'_1) - (\Gg_{\lambda,\ep} f)(v'_2)$ is arbitrarily small when $|v'_1-v'_2|$ is small enough, uniformly in $f\in\Xx$ such that $\|f\|_\Xx \leq 1$, since $\rho_\ep$, $\gamma$ and $g_a$ are uniformly continuous.
Therefore by the Ascoli-Arzel\`a theorem for any $\lambda\geq 0$ the family $\big( \Gg_{\lambda,\ep} (f) , f\in \Xx \big)$ is  compact in $\Xx$. 

\vip
\noindent {\it Step 3.} We now study the mapping $\lambda \leadsto \mu_{\lambda,\ep}$. Our aim is to prove that there exists $\lambda_\ep >0$ such that
\begin{equation} \label{4_eq:mulambdaep1}
\mu_{\lambda_\ep,\ep} = 1.
\end{equation}
To prove so we successively verify that 
{\bf 1)} as $\lambda$ increases, $\mu_{\lambda,\ep}$ does not increase,
{\bf 2)} the mapping $\lambda \leadsto \mu_{\lambda,\ep}$ is continuous,
{\bf 3)} for $\lambda = 0$, $\mu_{\lambda = 0,\ep} = 2(1+\ep)>1$,
{\bf 4)} as $\lambda \rightarrow \infty$, $\mu_{\lambda,\ep}$ converges to zero.
\vip
To prove {\bf 1)}, since $\mu_{\lambda,\ep}$ is the spectral radius of $\Gg_{\lambda,\ep}$, by the Gelfand-Beurling formula, it holds 
\begin{equation} \label{4_eq:firstvp}
\mu_{\lambda,\ep} = \lim_{n\rightarrow \infty} \interleave \Gg_{\lambda,\ep}^{n}  \interleave^{1/n}
\end{equation}
where
$$ \interleave \Gg_{\lambda,\ep}^{n}  \interleave = \sup_{f \in \Xx, \, \|f\|_\Xx = 1} \|\Gg_{\lambda,\ep}^nf \|_{\Xx} = \sup_{f \in \Xx, \, f\geq0, \, \|f\|_\Xx = 1} \|\Gg_{\lambda,\ep}^nf \|_{\Xx}.$$
Note that $\Gg_{\lambda,\ep}$ itself decreases as $\lambda$ increases: if $\lambda_2 > \lambda_1$ then $\Gg_{\lambda_2,\ep}(f) < \Gg_{\lambda_1,\ep}(f)$ for any nonnegative $f\in\Xx$. Also note that, for two nonnegative functions $f\in \mathcal X$ and $g\in \mathcal X$, if $f<g$ then $\Gg_{\lambda,\ep}(f) < \Gg_{\lambda,\ep}(g)$ for any $\lambda\geq0$. Relying on the two previous facts, one easily checks that if $\lambda_2 > \lambda_1$ then $\Gg_{\lambda_2,\ep}^n(f) < \Gg_{\lambda_1,\ep}^n(f)$ for any nonnegative $f\in\Xx$ and any integer $n$. Using \eqref{4_eq:firstvp}, we deduce $\mu_{\lambda_1,\ep} \leq \mu_{\lambda_2,\ep}$.
{\bf 2)} Let $(\lambda_k)_{k\geq 0}$ be a sequence such that $\lambda_k\geq0$ converges to $\bar \lambda \geq 0$ as $k\rightarrow \infty$. 
We know there exists a unique positive $U_{\lambda_k,\ep} \in \Xx$ satisfying \eqref{4_eq:mulambdaep} such that $\|U_{\lambda_k,\ep} \|_{\Xx}=1$. 
Also, there exists a unique positive $U_{\bar \lambda,\ep} \in \Xx$ satisfying \eqref{4_eq:mulambdaep} such that $\|U_{\bar \lambda,\ep} \|_{\Xx}=1$. 
Since $ \Gg_{\lambda_k,\ep} U_{\bar \lambda,\ep} = \mu_{\bar \lambda,\ep} U_{\bar \lambda,\ep} + (\Gg_{\lambda_k,\ep} - \Gg_{\bar \lambda,\ep})U_{\bar \lambda,\ep}$ 
and since $\|(\Gg_{\lambda_k,\ep} - \Gg_{\bar \lambda,\ep})U_{\bar \lambda,\ep}\|_{\Xx} \leq \eta_k \|U_{\bar \lambda,\ep} \|_{\Xx}$ for some small $\eta_k$ (recall that $\lambda_k\rightarrow\bar \lambda$), 
we deduce that $\mu_{\bar \lambda,\ep} U_{\bar \lambda,\ep}  -\eta_k \leq \Gg_{\lambda_k,\ep} U_{\bar \lambda,\ep} \leq \mu_{\bar \lambda,\ep} U_{\bar \lambda,\ep}  + \eta_k$. 
Now we have $0< \min U_{\bar \lambda,\ep} \leq U_{\bar \lambda,\ep} \leq 1$. 
Thus 
$$
\big( \mu_{\bar \lambda,\ep} -\eta_k\big) U_{\bar \lambda,\ep}   \leq \Gg_{\lambda_k,\ep} U_{\bar \lambda,\ep} \leq \big( \mu_{\bar \lambda,\ep}  + \tfrac{\eta_k}{\min U_{\bar \lambda,\ep} } \big) U_{\bar \lambda,\ep}.
$$
It follows that $$ \mu_{\bar \lambda,\ep} -\eta_k  \leq \mu_{\lambda_k,\ep} \leq  \mu_{\bar \lambda,\ep}  + \tfrac{\eta_k}{\min U_{\bar \lambda,\ep}}.$$ 
As $k\rightarrow\infty$, $\eta_k$ goes to zero and thus $ \mu_{\bar \lambda_k,\ep}\rightarrow \mu_{\bar \lambda,\ep}$, which proves the continuity.
To prove {\bf 3)}, we successively compute,
$$
(\Gg_{\lambda = 0,\ep} f)(v') = 2 \iint_{\Ss} f(v) \Psi_{\gamma/g_a}(a,v) \rho_\ep(v,v') dv da = 2 \int_{\Vv} f(v) \rho_\ep(v,v') dv
$$
since $\int_0^\infty  \Psi_{\gamma/g_a}(a,v) da = 1$ for any $v\in \Vv$, and
$$
\int_\Vv (\Gg_{\lambda = 0,\ep} f)(v')  dv' = 2(1+\ep) \int_{\Vv} f(v) dv
$$
since $\int_\Vv \rho_\ep(v,v') dv' = 1+\ep$ for any $v\in\Vv$.
Thus, by \eqref{4_eq:mulambdaep} and choosing $f=U_{\lambda,\ep}$ in the previous calculus,
$$
\int_{\Vv} (\Gg_{\lambda = 0,\ep} U_{\lambda = 0,\ep})(v') dv' = \mu_{\lambda = 0,\ep} \int_\Vv U_{\lambda = 0,\ep}(v') dv' = 2(1+\ep) \int_\Vv U_{\lambda = 0,\ep}(v) dv.
$$
Since $\int_\Vv U_{\lambda = 0,\ep}(v) dv < \infty$ ($U_{\lambda = 0,\ep}$ being bounded and $\Vv$ being a compact set), we deduce that $\mu_{\lambda  = 0} =2(1+\ep)$.
To prove {\bf 4)},
\begin{align*}
\int_\Vv (\Gg_{\lambda,\ep}f)(v') dv' & =  2(1+\ep)\iint_\Ss f(v) \exp\Big( - \int_0^a \frac{\lambda}{g_a(s,v)} ds\Big)  \Psi_{\gamma/g_a}(a,v) dvda \\
& \leq 2(1+\ep) \frac{ |g_a|_\infty |\Psi_{\gamma/g_a}|_\infty}{\lambda}  \int_\Vv f(v) dv
\end{align*}
using again $\int_\Vv \rho_\ep(v,v')dv' = 1+\ep$ and Assumption~\ref{4_ass:age_basic} for the upper bound. 
Then, as previously, by \eqref{4_eq:mulambdaep} and taking $f=U_{\lambda,\ep}$ in the previous calculus, we check that
\begin{equation} \label{4_eq:muupperbound}
\mu_{\lambda,\ep} \leq 2(1+\ep) \frac{|g_a|_\infty |\Psi_{\gamma/g_a}|_\infty}{\lambda} ,
\end{equation}
which implies that $\mu_{\lambda,\ep}\rightarrow 0$ as $\lambda \rightarrow \infty$.

\vip
\noindent {\it Step 4.} 
In this last step, the aim is to let $\ep$ tend to zero. Let a $\lambda_\ep$ be defined by \eqref{4_eq:mulambdaep1} and denote by $U_{\lambda_\ep,\ep} = U _\ep\in \Xx$ the associated positive eigenvector such that $\| U_\ep \|_\Xx=1$. 
On the one hand, the family $(\lambda_\ep, 0<\ep<1)$ is bounded, recalling \eqref{4_eq:mulambdaep1} and \eqref{4_eq:muupperbound}. 
On the other hand, the family, $(U_\ep , 0<\ep<1)$ is compact in $\Xx$ (recall that $U_\ep = \Gg_{\lambda_\ep,\ep} U_\ep$ and use again the Ascoli-Arzel\`a theorem -- note that we achieve uniformity in $\varepsilon \in (0,1)$ using the fact that $(\lambda_\ep, 0<\ep<1)$ is bounded). 
Thus we can extract a subsequence, still denoted by $(\lambda_\ep,U_\ep)$, converging to $(\bar \lambda, \bar U)$ in $\RR\times \Xx$ with $\bar \lambda \geq0$ and $\bar U \in\Xx$ positive such that $\|\bar U \|_\Xx=1$. Since
$$
U_\ep(v') = 2 \iint_{\Ss} U_\ep(v) \exp\Big( - \int_0^a \frac{\lambda_\ep}{g_a(s,v)} ds\Big)  \Psi_{\gamma/g_a}(a,v) (\rho(v,v') + |\Vv|^{-1}\ep) dv da,
$$
letting $\ep\rightarrow0$, by the dominated convergence theorem, we obtain $\bar U = \Gg_{\bar \lambda} \bar U$, which means that we have found a solution to \eqref{4_eq:lambdasolof}.
Now set $\lambda_{\gamma,\rho} = \bar \lambda$ and for any $(a,v)\in\Ss$,
$$
N_{\gamma,\rho} (a,v)= \frac{\kappa \bar U(v)}{g_a(a,v)} \exp\Big( - \int_0^a \frac{\bar \lambda + \gamma(s,v)}{g_a(s,v)} ds \Big),
$$
reminding \eqref{4_eq:eigenvect} and $U_{\gamma,\rho} = g N_{\gamma,\rho}$, with $\kappa>0$ chosen such that $\iint_\Ss N_{\gamma,\rho} = 1$ (which is possible since $\iint_\Ss N_{\gamma,\rho}  < \infty$).
\end{proof}

\begin{proof}[Adjoint eigenproblem] The proof follows the same steps as in the direct eigenproblem.
\vip
\noindent {\it Step 1.} Since $\phi_{\gamma,\rho}$ satisfies \eqref{4_eq:EDP_Psi}, one easily checks that
$$
\frac{\p}{\p a} \Big( \phi_{\gamma,\rho}(a,v) e^{-\int_0^a \frac{\lambda_{\gamma,\rho}+\gamma(s,v)}{g_a(s,v)} ds}\Big) = -2 \Psi_{\gamma/g_a}(a,v) e^{-\int_0^a \frac{\lambda_{\gamma,\rho}}{g_a(s,v)} ds}  \int_\Vv \phi_{\gamma,\rho}(0,v') \rho(v,v') dv' 
$$
with $\Psi_{\gamma/g_a}$ defined by \eqref{4_eq:fgammasurg}. Integrating in $a$ between zero and infinity, we deduce
\begin{equation} \label{4_eq:lambdasolof2} 
\phi_{\gamma,\rho}(0,v) = 2\int_0^\infty \Psi_{\gamma/g_a}(a,v) e^{-\int_0^a \frac{\lambda_{\gamma,\rho}}{g_a(s,v)} ds} da  \int_\Vv \phi_{\gamma,\rho}(0,v') \rho(v,v') dv' 
\end{equation}
and integrating between zero and $a$, we deduce
\begin{equation}  \label{4_eq:eigenvect2}
\phi_{\gamma,\rho}(a,v) = 2 e^{\int_0^a \frac{\lambda_{\gamma,\rho}+\gamma(s,v)}{g_a(s,v)} ds}  \int_a^\infty \Psi_{\gamma/g_a}(s,v) e^{-\int_0^s \frac{\lambda_{\gamma,\rho}}{g_a(t,v)} dt} ds  \int_\Vv \phi_{\gamma,\rho}(0,v') \rho(v,v') dv'.
\end{equation}
\begin{rk} \label{4_eq:phistrictpositif}
By a \emph{reductio ad absurdum} argument, using \eqref{4_eq:lambdasolof2} and the Markov kernel properties, we prove that for any $v\in\Vv$, $\phi_{\gamma,\rho}(0,v) >0$. Then, using \eqref{4_eq:eigenvect2}, we deduce that $\phi_{\gamma,\rho}(a,v)>0$ for any $(a,v) \in \Ss$.
\end{rk}
Equation \eqref{4_eq:lambdasolof2} leads us to define an operator $\Gg^*_{\lambda} : \Xx \rightarrow \Xx$ by
\begin{equation*} \label{4_eq:Glambda}
(\Gg^*_{\lambda} f)(v) = 2\int_0^\infty \Psi_{\gamma/g_a}(a,v) e^{-\int_0^a \frac{\lambda}{g_a(s,v)} ds} da  \int_\Vv f(v') \rho(v,v') dv' , \quad v \in \Vv,
\end{equation*}
for any $\lambda\geq0$. The aim is now to find a solution $(\lambda,f)$ to the equation $\Gg^*_{\lambda} (f) = f$.
\vip
\noindent {\it Step 2.} For a fixed $\ep>0$, we define a regularized operator $\Gg^*_{\lambda,\ep} : \Xx \rightarrow \Xx$ by
\begin{equation*} \label{4_eq:Glambda}
(\Gg^*_{\lambda,\ep} f)(v) = 2\int_0^\infty \Psi_{\gamma/g_a}(a,v) e^{-\int_0^a \frac{\lambda}{g_a(s,v)} ds} da  \int_\Vv f(v') \rho_\ep(v,v') dv' , \quad v \in \Vv,
\end{equation*}
with $\rho_\ep$ picked as in \eqref{4_eq:rhoep}, for any $\lambda \geq 0$. With similar arguments as previously, we prove we are in position to apply the Krein-Rutman theorem: for any $\lambda\geq0$ there exist a unique $\mu_{\lambda,\ep}>0$ and a unique positive $H_{\lambda,\ep}\in \Xx$ such that 
\begin{equation} \label{4_eq:mulambdaep2}
\Gg_{\lambda,\ep}^*(H_{\lambda,\ep}) = \mu_{\lambda,\ep} H_{\lambda,\ep}
\end{equation}
and $\|H_{\lambda,\ep}\|_\Xx = 1$.

\vip
\noindent {\it Step 3.} The study of $\lambda \leadsto \mu_{\lambda,\ep}$ consists in proving the same four points as previously. Only the verification of {\bf 3)} and {\bf 4)} slightly differs. To prove {\bf 3)}, we successively compute
$$
(\Gg_{\lambda = 0,\ep}^* f)(v) = 2 \int_\Vv f(v') \rho_\ep(v,v') dv' 
$$
and 
$$
\int_\Vv (\Gg_{\lambda = 0,\ep}^* f)(v)  dv = 2 \int_\Vv f(v') \Big(\int_\Vv \rho(v,v') dv +\ep \Big) dv'  \geq 2 (\varrho+\ep) \int_\Vv f(v') dv',
$$
relying on Assumption~\ref{4_ass:age_rho}.
Thus, choosing $f=H_{\lambda,\ep}$ in the previous calculus and using \eqref{4_eq:mulambdaep2}, we obtain $\mu_{\lambda = 0, \ep} \geq 2(\varrho+\ep)>1$ as soon as $\varrho>1/2$. To prove {\bf 4)}, we readily obtain that 
$$
\mu_{\lambda,\ep} \leq 2(|\Vv| |\rho|_\infty+\ep) \frac{|g_a|_\infty |\Psi_{\gamma/g_a}|_\infty}{\lambda} \rightarrow 0
$$
as $\lambda \rightarrow 0$.
\vip
\noindent {\it Step 4.}  We let $\ep$ go to zero as previously and we find $(\bar \lambda^*,\bar H)$ with $\bar \lambda^* \geq 0$ and $\bar H$ non-negative, $\|\phi^*\|_\Xx=1$, such that $\Gg_{\bar \lambda^*}^* \bar H = \bar H$, which means we have found a solution to \eqref{4_eq:lambdasolof2}. Recalling \eqref{4_eq:eigenvect2}, we set
$$
\phi_{\gamma,\rho}(a,v) = 2 \kappa' e^{\int_0^a \frac{\bar \lambda^*+\gamma(s,v)}{g_a(s,v)} ds}  \int_a^\infty \Psi_{\gamma/g_a}(s,v) e^{-\int_0^s \frac{\bar \lambda^*}{g_a(t,v)} dt} ds  \int_\Vv \bar H(v') \rho(v,v') dv'.
$$
and fix $\kappa'>0$ such that $\iint_\Ss N_{\gamma,\rho} \phi_{\gamma,\rho} = 1$.
\end{proof}

\begin{proof}[Uniqueness of the eigenelements] We successively prove the uniqueness of the eigenvalue, of the direct eigenvector and of the adjoint eigenvector.
\vip
\noindent {\it Step 1.} Let $(\lambda,N)$ be a solution to the direct eigenproblem \eqref{4_eq:EDP_N} and $(\lambda^*,\phi)$ be a solution to the adjoint eigenproblem \eqref{4_eq:EDP_Psi}. We first prove that $\lambda = \lambda^*$. Indeed,
\begin{align*}
\lambda \iint_\Ss N \phi & = \iint_\Ss \Big( - \frac{\p}{\p a} (g_aN)(a,v) - \gamma(a,v) N(a,v) \Big) \phi(a,v) dvda \\
& =  \iint_\Ss N(a,v) \Big( g_a(a,v)\frac{\p}{\p a} \phi(a,v) + \gamma(a,v) \big( 2  \int_{\Vv} \phi(0,v') \rho(v,v') dv' - \phi(a,v) \big) \Big) dvda \\
& = \lambda^* \iint_\Ss N \phi
\end{align*}
and since $ \iint_\Ss N \phi > 0$ we deduce $\lambda= \lambda^*$.
\vip
\noindent {\it Step 2.} Let $(\lambda,N_1)$ and $(\lambda,N_2)$ be two solutions of the direct eigenproblem \eqref{4_eq:EDP_N}. We prove that $N_1 = N_2$. Following the proof of Proposition 6.3 of \cite{4_Perthame}, we prove that $\widetilde N = |N_1 - N_2|$ satisfies
$$
\iint_\Ss \Big( \frac{\p}{\p a} (g_a\widetilde N)(a,v) + \gamma(a,v) \widetilde N(a,v) \Big) \phi(a,v) dvda = 0
$$ 
taking $\phi$ a solution to \eqref{4_eq:EDP_Psi} as a test function.
We deduce that
\begin{multline*}
2 \int_\Vv \Big( \iint_\Ss \gamma(a,v) \big|N_1-N_2\big|(a,v) \rho(v,v') dvda \Big) \phi(0,v') dv' \\= 2 \int_\Vv \Big| \iint_\Ss \gamma(a,v) (N_1-N_2)(a,v) \rho(v,v') dvda \Big| \phi(0,v') dv',
\end{multline*}
using that $\phi$ is a solution to \eqref{4_eq:EDP_Psi} and since both $N_1$ and $N_2$ satisfy the boundary condition of the eigenproblem \eqref{4_eq:EDP_N}.
Thanks to the fact that $\phi(0,v')>0$ for $v' \in \Vv$, we deduce that $\gamma(a,v) (N_1-N_2)(a,v) \rho(v,v')$ is of constant sign. Then, integrating in $v'$, $\gamma(a,v) (N_1-N_2)(a,v)$ is also of constant sign. Recall that for each rate $v\in \Vv$, the division rate $\gamma(a,v)$ is positive for $a$ belonging to some $[a_{\min}(v),a_{\max}(v)]$, thus $(N_1-N_2)(a,v)$ is of constant sign on $\{[a_{\min}(v),a_{\max}(v)]\times\{v\},v\in\Vv\}$. Using \eqref{4_eq:eigenvect}, we deduce that $(N_1-N_2)(0,v)$ is of constant sign on $\Vv$ and thus, using \eqref{4_eq:eigenvect} again, $(N_1-N_2)(a,v)$ is of constant sign on $\Ss$.
Since we have $\iint_\Ss (N_1-N_2) = 0$, the conclusion $N_1 = N_2$ follows.
To conclude,
Fredholm alternative (see \cite{4_Brezis}) ensures that uniqueness of a solution to \eqref{4_eq:EDP_N} implies uniqueness of a solution to \eqref{4_eq:EDP_Psi}.
\end{proof}
The proof of Theorem~\ref{4_thm:ageeig_general} is now complete.
\end{proof}

\section{Appendix: Supplementary result on the age-structured model} \label{sec:supp}

In this section, working under {\bf Model (A+V)}, our aim is to compare the growth speed of the two populations defined on page~\pageref{page:pop} when preserving the mean lifetime (instead of the mean aging rate as previously). In other words, our aim is to compare $\lambda_{B,\bar v}$ to $\lambda_{B,\rho}$ for a density $\rho$ on $\Vv$ such that 
\begin{equation} \label{eq:preserve}
\int_{\Vv} \frac{1}{v} \rho(v) dv = \frac{1}{\bar v}.
\end{equation}

In order to check  this fact, let us consider the stochastic description of {\bf Model (A+V)}. Let $\tau_u$ be the aging rate of a cell $u$ such that $\mathbb P(\tau_u \in dv) = \rho(v)dv.$
 The physiological age of cell $u$ is $\Delta_u = \tau_u \zeta_u$ with $\zeta_u$ its lifetime such that 
 $$
 \mathbb P \big(\zeta_u \in [t, t+dt) | \zeta_u \geq t, \tau_u = v \big) = \gamma(vt,v) dt = vB(vt) dt,
 $$ 
 since $vt = a$ is the physiological age and since $\gamma(a,v) = vB(a)$. 
 In addition all cells are independent.
One can check that the density of the physiological age $\Delta_u = \tau_u \zeta_u $ is 
 $$
  \mathbb P (\Delta_u \in da) = B(a) \exp(-\int_0^a B(s) ds) da = \Psi_B(a),
 $$
 which means that $\Delta_u$ is independent of $\tau_u$.
As a consequence
 $$
 \mathbb E[\zeta_u]  = \mathbb E\bigg[\frac{\Delta_u}{\tau_u}\bigg] =   \mathbb E[\Delta_u] \times  \mathbb E\bigg[\frac{1}{\tau_u}\bigg] = \int_{0}^\infty a \Psi_B(a) da \times \int_{\mathcal V} \frac{1}{v} \rho(v) dv.
$$
On the other hand, when there is no variability, the mean lifetime is $$\int_{0}^\infty a \Psi_B(a) da \times \frac{1}{\bar v}.$$

Note that the influence of the division rate $B$ and the influence of the aging rates distribution on the mean lifetime are independent.
Thus, for a fixed $B$, preserving the mean lifetime of the cells of the two populations is indeed equivalent to \eqref{eq:preserve}.\\

\begin{thm} \label{4_thm:ageinfluencebis}
Consider {\bf Model (A+V)} with Specifications \eqref{4_eq:gv}, \eqref{4_eq:vBa}, \eqref{4_eq:noheredity} and \eqref{eq:preserve}. Then 
$$
\lambda_{B,\rho} > \lambda_{B,\bar v}
$$
for any division rate $B$.
\end{thm}

\begin{proof}
One wants to compare  {\bf 1)} the Malthus parameters $\lambda_{B,\bar v}$ solution to  \eqref{4_eq:age_lambda_novar} ; to {\bf 2)} the Malthus parameters $\lambda_{B,\rho}$ solution to  \eqref{4_eq:age_lambda} which is equivalent to
\begin{equation*} \label{4_eq:age_lambdabis}
2 \iint_{\tilde \Ss} \exp \big( - \lambda_{B,\rho} a \tilde v \big) \Psi_B(a) \tilde \rho(\tilde v) d\tilde vda = 1 \,\, ;
\end{equation*}
setting $\tilde \Ss = (0,\infty)\times 1/\Vv$ and $\tilde \rho(\tilde v) = \rho(1/\tilde v)/\tilde v^2$. 

One can easily check that $\tilde \rho$ is of mean $1/\bar v$. Then the comparison of $\lambda_{B,\bar v}$  to $\lambda_{B,\rho}$ immediately follows using Jensen's inequality following the same method as in the proof of Theorem \ref{4_thm:ageinfluence}.
\end{proof}

Note that one can also obtain an analog of Theorem \ref{4_thm:ageperturbation} and this is left to the reader.
The conclusions of Theorem \ref{4_thm:ageinfluence} and \ref{4_thm:ageinfluencebis} differ since the preserved quantities between the two populations are not the same in the two cases (mean aging rate {\it vs.} mean lifetime). Both are mathe\-matically interesting, even if, in view of applications concerning  the bacteria {\it E. coli}, {\bf Model~(S+V)} would be useful. 

\vip
\vip

\noindent {\bf Acknowledgements.} I thank L. Robert for suggesting this very interesting biological problem and M. Doumic for many helpful discussions.  I am also grateful to P. Reynaud-Bouret for a careful reading and to the referee for suggestions in order to improve the manuscript.

\end{document}